\documentclass{article}
\usepackage{amsmath}
\usepackage{amssymb}
\usepackage{geometry}
\usepackage{amsthm}
\usepackage{mathrsfs}
\usepackage{indentfirst}
\usepackage{cases}
\newtheorem{thm}{Theorem}[section]
\newtheorem{lem}[thm]{Lemma}
\newtheorem{cor}[thm]{Corollary}
\newtheorem{defi}[thm]{Definition}
\newtheorem{prop}[thm]{Proposition}
\newtheorem{rem}[thm]{Remark}
\numberwithin{equation}{section}
\title{The viscous damping of three dimensional spherical gas bubble  inside unbounded compressible liquid}
\author{Lifeng Zhao\footnote{School of Mathematical Sciences, University of Science and Technology of China, Hefei, Anhui, 230026, PR China, zhaolf@ustc.edu.cn},\quad Liangchen Zou\footnote{School of Mathematical Sciences, University of Science and Technology of China, Hefei, Anhui, 230026, PR China, zlc0601@mail.ustc.edu.cn}}
\date{}
\begin{document}
\maketitle
\begin{abstract}
The present paper considers a homogeneous bubble inside an unbounded polytropic compressible liquid with viscosity. The system is governed by the Navier-Stokes equation with free boundary which is determined by the kinematic and dynamic boundary conditions on the bubble-liquid interface. The global existence of solution is proved, and the $\dot{H}^1$ asymptotic stability of the spherical equilibrium in terms of viscous damping together with a explicit decay rate is given in bare energy methods. 
\end{abstract}
\section{Introduction}
The bubble-liquid system is omnipresent in nature, and has prevalent applications in different fields. Examples include microbubble ultrasound contrast agents \cite{MUCA}, the damage to ships caused by underwater explosions \cite{RSUE}\cite{SRSS}, the bubble dynamics in magmas \cite{DBFGM}, the influence of cavitation on ship propellers \cite{CRSPD}, etc. For a collection of bubble phenomenons and applications, one can refer to the review article \cite{FSTS} by Leighton.\\
\indent The research of bubble dynamics can be traced to Rayleigh's study \cite{RAY} of spherical homogeneous gas bubble in an incompressible, inviscid liquid with surface tension, which investigated the pressure during the cavity collapse. In the incompressible, spherical symmetry case, the dynamics of the bubble-liquid system is then reduced to the well-known Rayleigh-Plesset equation. However, Rayleigh-Plesset equation failed to explain the damped oscillation of underwater explosion bubble, which was found caused by the compressibility. To this end, Keller \cite{1956Damping} modified the Rayleigh-Plesset equation by introducing a wave context. Rayleigh-Plesset equation and Keller equation have been widely studied by both numerical and mathematical methods in a great variety of settings. For a systematic overview of Rayleigh-Plesset equation, one can refer to Ohnawa and Suzuki \cite{ff07c1facba84951a1479c949d23e506} and the references therein, which investigated Rayleigh-Plesset equation and Keller equation mathematically, and presented related numerical results . \\ 
\indent When compressibility, nonlinearity and asymmetry are taken into considerations, the analysis to the bubble-liquid system becomes complicated, and the bubble-liquid system is then described by compressible Euler or Navier-Stokes equations depending on whether viscosity is considered or not on exterior domains with free boundaries. Shapiro and Weinstein \cite{RDBO} described the dynamics of homogeneous bubble surrounded by a compressible, inviscid liquid with surface tension and proved exponential radiative decays in linear approximation near the spherical equilibrium by using spherical harmonics decomposition. For inhomogeneous bubble, the recent work of Lai and Weinstein \cite{FBPGB} proved the asymptotic stability of spherical equilibrium provided the liquid external to the bubble is incompressible.\\
\indent Compared to the incompressible, spherical symmetry case, where viscosity contributes nothing to the liquid external to the gas bubble, viscosity plays an important role in the compressible setting. It was known pretty early that the radius of a pulsing gas bubble in a liquid undergoes damping induced by various mechanisms including thermal effects, energy radiated outward by sound waves and the energy lost due to viscosity \cite{doi:10.1121/1.1907675}, see also the resent book \cite{ZHANG2022131}. The present paper will focus on the viscous damping, and we consider a homogeneous gas bubble surrounded by a compressible viscous liquid with surface tension and spherical symmetry. The bubble-liquid system consists of three parts: the external liquid, the gas bubble within and the bubble-liquid interface. The liquid is governed by Navier-Stokes equations. The pressure of the homogeneous bubble is assumed to satisfy the polytropic gas law. The interface is determined by the kinematic and dynamic boundary conditions related to the liquid and bubble pressure together with the surface tension. Therefore, as a whole, the bubble-liquid system is determined by the equation system
\begin{numcases}
{} \partial_t\rho+\nabla\cdot(\rho u)=0, &$\xi\in\Omega(t)^c,\;t>0,$\\
\rho\partial_tu+\rho u\cdot\nabla u+\nabla p=\mu\nabla\cdot D(u), &$\xi\in\Omega(t)^c,\;t>0,$\\
\partial_t\Xi(z,t)=u(\Xi(z,t),t)\cdot n(\Xi(z,t),t), &$z\in\mathbb{S},\;t>0,$\\
p(\Xi(z,t),t)-\mu D(u)(\Xi(z,t),t)=p_b(t)-2\sigma H[\Xi], &$z\in\mathbb{S},\;t>0,$
\end{numcases}
where $u$, $\rho$, $p$ denote the velocity, density and pressure of the liquid external to the bubble; $\Omega(t)\subset\mathbb{R}^3$ is the space occupied by the bubble; $D(u):=\frac{1}{2}\left(\nabla u+\nabla u^T\right)$ is the stress tensor; the viscosity coefficient $\mu$ is assumed to be a positive constant. Here the bubble surface is assumed to be diffeomorphism to the unit sphere $\mathbb{S}$ through $\Xi$. $n(\Xi,t)$ denotes the outer normal vector at the $\Xi(z,t)$ on the bubble surface. $\sigma$ is the surface tension, and $H[\Xi]:=\frac{1}{2}\nabla\cdot n$ is the mean curvature at $\Xi$. The pressure of the liquid $p$ is assumed polytropic, namely, $p=C_0\rho^\gamma$ for $\gamma>1$. As mentioned above, the bubble pressure $p_b$ is assumed homogeneous and satisfies the polytropic gas law: $p_b=C_1|\Omega(t)|^{-\gamma_0}$ for $\gamma_0>1$.\\
\indent Since we restrict the study to the spherical symmetry setting, suppose that $$\rho(\xi,t)=\rho(r,t),\;u(\xi,t)=u(r,t)\frac{\xi}{r},\; \Xi(z,t)=R(t)\frac{\xi}{r}\text{, with } r=|\xi|.$$ Then the outer normal vector $n=\frac{\xi}{r}$, and the mean curvature $H[\Xi]=R^{-1}$. Hence in the spherical case, system (1.1-1.4) becomes
\begin{numcases}
{} \partial_t\rho+r^{-2}\partial_r(\rho u)=0, &$r>R(t),\;t>0,$\\
\rho\partial_t u+\rho u\partial_r u+\partial_r p=\mu\partial_r(\partial_r+\frac{2}{r})u, &$r>R(t),\;t>0,$\\
\frac{dR}{dt}=u|_{r=R(t)}, &$t>0,$\\
(p-\mu\partial_r u)|_{r=R(t)}=p_b-2\sigma R^{-1}, &$t>0.$
\end{numcases} 
The system (1.5-1.8) admits an equilibrium state, and after nondimensionalization \cite[Appendix C]{RDBO}, one can assume the equilibrium state to be $$\rho=1,\;u=0,\;R=1,\; p=\frac{Ca}{2}\rho^\gamma,\;p_b=\left(\frac{Ca}{2}+\frac{2}{We}\right)R^{-3\gamma_0},$$ where $Ca$ is called the cavitation number, and $We$ is the Weber number. The equations (1.5-1.8) are then rewritten as
\begin{numcases}
{} \partial_t\rho+r^{-2}\partial_r(\rho u)=0, &$r>R(t),\;t>0,$\\
\rho\partial_t u+\rho u\partial_r u+\frac{Ca}{2}\partial_r(\rho^\gamma) =\mu\partial_r(\partial_r+\frac{2}{r})u, &$r>R(t),\;t>0,$\\
\frac{dR}{dt}=u|_{r=R(t)}, &$t>0,$\\
(\frac{Ca}{2}\rho^\gamma-\mu\partial_r u)|_{r=R(t)}=\left(\frac{Ca}{2}+\frac{2}{We}\right)R^{-3\gamma_0}-\frac{2}{We}R^{-1}, &$t>0.$
\end{numcases} 
\indent The system (1.9)-(1.12) is a free boundary problem with a nonlinear boundary condition (1.12), so it is natural to introduce the Lagrangian coordinates. Namely, 
define the Lagrangian coordinate $x:=\int_{R(t)}^r\rho(s,t) s^2ds$. Physically, $x$ stands for the mass of liquid external to the bubble but inside a spherical domain with radius $r$. Then using (1.9), a direct calculation gives that 
\begin{equation}
\left[\begin{matrix}
\frac{\partial x}{\partial r} & \frac{\partial x}{\partial t}\\
\frac{\partial t}{\partial r} & \frac{\partial t}{\partial t}
\end{matrix}\right]
=\left[\begin{matrix}
\rho r^2 & -\rho r^2 u\\
0 & 1
\end{matrix}\right],\;
\left[\begin{matrix}
\frac{\partial r}{\partial x} & \frac{\partial r}{\partial t}\\
\frac{\partial t}{\partial x} & \frac{\partial t}{\partial t}
\end{matrix}\right]
=\left[\begin{matrix}
(\rho r^2)^{-1} & u\\
0 & 1
\end{matrix}\right].
\end{equation}
In view of (1.13), the system (1.9-1.12) is transformed to:
\begin{numcases}
{} \partial_t\rho+\rho^2\partial_x(r^2u)=0, &$x>0,\;t>0,$\\
\partial_t u+\frac{Ca}{2}r^2\partial_x(\rho^\gamma) =\mu r^2\partial_x\left(\rho\partial_x(r^2u)\right), &$x>0,\;t>0,$\\
\frac{dR}{dt}=u|_{x=0}, &$t>0$\\
(\frac{Ca}{2}\rho^\gamma-\mu\rho r^2\partial_x u)|_{x=0}=\left(\frac{Ca}{2}+\frac{2}{We}\right)R^{-3\gamma_0}-\frac{2}{We}R^{-1}, &$t>0,$\\
r=\left(R(t)^3+3\int_0^x\rho^{-1}(y,t)dy\right)^{\frac{1}{3}}=r(x,0)+\int_0^t u(y,\tau)d\tau, &$x>0,\;t>0,$
\end{numcases} \\
with the initial value \begin{equation}(u,\;\rho,\;R)|_{t=0}=(u_0,\;\rho_0,\;R_0),\end{equation} and compatibly \begin{equation} r_0(x)=\left(R_0^3+3\int_0^x\rho_0^{-1}(y)dy\right)^{\frac{1}{3}}.\end{equation}
\indent The first result is the global existence and the uniqueness of the generalized solution to (1.14-1.19), which is defined as following: 
\begin{defi}
$(u,\;\rho,\;R)$ is said to be a generalized solution to system (1.14-1.19) on $[0, T]$ with initial value $(u_0,\;\rho_0,\;R_0)$, if 
$$u\in C\left([0,T],\;L^2(0,+\infty)\right),\;r^2\partial_x u\in C\left([0,T],\;L^2(0,+\infty)\right),$$
$$\partial_t u\in L^\infty\left([0,T],\;L^2(0,+\infty)\right),\;r^2\partial_t\partial_xu\in L^2\left([0,T],\;L^2(0,+\infty)\right),$$
$$\rho-1\in C\left([0,T],\;L^2(0,+\infty)\right),\;r^2\partial_x(\log\rho)\in C\left([0,T],\;L^2(0,+\infty)\right),$$
$$\partial_t\rho\in L^\infty\left([0,T],\;L^2(0,+\infty)\right),\;r^2\partial_t\partial_x(\log\rho)\in L^\infty\left([0,T],\;L^2(0,+\infty)\right),$$
$$\inf_{(x,t)\in (0,+\infty)\times [0,T]}\rho>0,\; \inf_{t\in [0,T]}R>0,\;\rho\in L^\infty\left([0,T],\;L^\infty(0,+\infty)\right),\;R\in L^\infty[0,T],$$
and (1.14)(1.15) are satisfied in the $L^\infty\left([0,T],\;L^2(0,+\infty)\right)$ sense while (1.16)(1.17) are satisfied in the trace sense.
\end{defi}
Now, we are in the position to state the main results:
\begin{thm}[Global existence and uniqueness]
Suppose that the initial value $(u_0,\;\rho_0,\;R_0)$ satisfies that $$u_0\in L^2(0,+\infty),\; r_0^2\partial_x u_0\in  L^2(0,+\infty),\;r_0^2\partial_x\left(\rho\partial_x(r^2u)\right)\in L^2(0,+\infty),$$ $$\int_0^\infty H(\rho_0)dx<+\infty,\;r_0^2\partial_x(\log\rho_0)\in  L^2(0,+\infty),\;\text{where } H(\rho)=\rho^{\gamma-1}-\gamma+(\gamma-1)\rho^{-1},$$
$$\inf\rho_0>0,\;\sup\rho_0\leq+\infty,\;0<R_0<\infty,\;\text{and } r_0\text{ is given by (1.20).}$$
Then there exists a unique global generalized solution to (1.14-1.19).
\end{thm}
\indent In contrast to free boundary problems for Navier-Stokes equations on bounded domains with vacuums (for example, \cite{LSDS}\cite{GWS1D}\cite{doi:10.1081/PDE-100002385}), the bubble pressure being positive avoids the formation of vacuums. However, the unboundedness of the domain also causes some ambiguity when establishing energy estimates, including the elliptic estimates and that the lack of decay of $u$ in space makes integration by parts ambiguous. To overcome these ambiguity, we borrow the idea from Jiang \cite{GSS} considering a related initial boundary value problem on bounded domains, and constructing approximate solutions to (1.14-1.19) using the solutions to this initial boundary problem on  bounded domains. 
\begin{thm}[Viscous damping]
Suppose that the initial value $(u_0,\;\rho_0,\;R_0)$ is close enough to the equilibrium state in the sense that for some small positive $\delta$
\begin{equation}
\|u_0\|_{L^2}^2+\int_0^\infty H(\rho_0)dx+\|r_0^2\partial_x(\log\rho_0)\|_{L^2}^2+(R_0-1)^2\leq\delta.
\end{equation} 
Then the global generalized solution given by Theorem 1.2 satisfies that 
\begin{equation}
\|r^2\partial_x u\|_{L^2}^2+\left\|\frac{u}{r}\right\|_{L^2}^2+\|r^2\partial_x\rho\|_{L^2}^2+(R-1)^2\leq C(1+t)^{-1},
\end{equation}
where $C$ is a constant depending on the initial data.
\end{thm}
\indent The proof of Theorem 1.3 requires a more careful estimate to bound the density $\rho$ from above and below uniformly in time by using the Bresch-Desjardins entropy estimate \cite{2DVSW} and making full use of the dissipations to cancel bad boundary terms. Note that compared with assumptions of Theorem 1.2 on regularities of the initial data, the smallness assumption (1.21) only applies on the low regularities $u_0$, $H(\rho_0)$, $R_0$, and $r_0^2\partial_x(\log\rho_0)$, which implies that a large gradient of the velocity in the initial data will not inhibit the resulted decay. The novelty is that system (1.14-1.19) involves a nonlinear boundary condition (1.17), and we avoid linearizing system (1.14-1.19) and work totally in the nonlinear scheme.\\
\indent In the following several sections, $Ca,\;We,\;\mu$ denote corresponding fixed constants. $c,\; C$ are used to denote constants only depending on the initial value and the above fixed constants. $c(T),\; C(T)$ are used to denote constants depending on the initial value, the above fixed constants and the time span $[0,T]$. For simplicity, sometimes $\partial_\alpha f$ is written as $f_\alpha$ for $\alpha=x,\;t,\;f=u,\;\rho$, etc. It is necessary to note that $c, c(T), C, C(T)$ are required independent on the size of the bounded domains.\\
\indent The plan of the paper is as follows. In the next section, we  state the related initial boundary value problem on bounded domains and prove the existence of global solutions to this related problem in a standard procedure: the short time existence, a-prior estimates and the continuity argument. In the first part of Section 3 the approximated solutions are constructed from the solutions on bounded domains and weak compactness is employed to obtain the exact solution to (1.14-1.19). Then the uniqueness is proved in the second part of Section 3. Finally, the uniform in time estimates are given in Section 4, which is then applied to obtain the viscous damping with the help of the differential inequality in Lemma 4.8.

\paragraph{Acknowledgements} L. Zhao is supported by NSFC Grant of China No. 12271497  and the National Key Research and Development Program of China  No. 2020YFA0713100.
\section{The bubble-liquid system on bounded domain}
 In this section, we temporarily abbreviate $L^\infty(0,k)$ as $L^\infty$, $L^2(0,k)$ as $L^2$, and correspondingly $\|\cdot\|_{L^\infty(0,k)}$ as $\|\cdot\|_{L^\infty}$, $\|\cdot\|_{L^2(0,k)}$ as $\|\cdot\|_{L^2}$. Now consider the bubble-liquid system on bounded domain $[0,k]$ for $k>1$, namely
\begin{numcases}
{} \partial_t\rho+\rho^2\partial_x(r^2u)=0, &$k>x>0,\;t>0,$\\
\partial_t u+\frac{Ca}{2}r^2\partial_x(\rho^\gamma) =\mu r^2\partial_x\left(\rho\partial_x(r^2u)\right), &$k>x>0,\;t>0,$\\
\frac{dR}{dt}=u|_{x=0},\; u|_{x=k}=0, &$t>0,$\\
(\frac{Ca}{2}\rho^\gamma-\mu\rho r^2\partial_x u)|_{x=0}=\left(\frac{Ca}{2}+\frac{2}{We}\right)R^{-3\gamma_0}-\frac{2}{We}R^{-1}, &$t>0,$\\
r=\left(R(t)^3+3\int_0^x\rho^{-1}(y,t)dy\right)^{\frac{1}{3}}=r(x,0)+\int_0^t u(y,\tau)d\tau, &$k>x>0,\;t>0,$\\
(u,\;\rho,\;R)|_{t=0}=(u_0,\;\rho_0,\;R_0), &$k>x>0,$\\
r_0(x)=\left(R_0^3+3\int_0^x\rho_0^{-1}(y)dy\right)^{\frac{1}{3}}, &$k>x>0.$
\end{numcases} 
\begin{prop}[Global existence on bounded domains]
Suppose the initial data $(u_0,\;\rho_0,\;R_0)$ satisfies that $$u_0\in L^2(0,k),\; r_0^2\partial_x u_0\in  L^2(0,k),\;r_0^2\partial_x\left(\rho_0\partial_x(r_0^2u_0)\right)\in L^2(0,k),$$ $$\int_0^k H(\rho_0)dx<+\infty,\;r_0^2\partial_x(\log\rho_0)\in  L^2(0,k),\;\inf_{x\in [0,+\infty)}\rho_0>0,\;\sup_{x\in [0,+\infty)}\rho_0\leq+\infty,\;0<R_0<\infty.$$ 
Then there exists a unique global generalized solution to (2.1-2.7). 
\end{prop}
\indent The proof of Proposition 2.1 includes the short time existence of solutions, a-priori estimates, and a standard continuity argument. The proof of the uniqueness is omitted here since it is the same as the uniqueness in the unbounded case, whose proof is given in Section 3. The following several lemmas in this section are devoted to establish the a-priori estimates. $(u,\;\rho,\;R)$ is assumed to be any generalized solution of (2.1-2.7) on $[0,T]$. We begin with the following basic energy identity.
\begin{lem}[Basic energy]
Introduce the notations
$$P(R)=\frac{1}{3\gamma_0-3}\left(\frac{Ca}{2}+\frac{2}{We}\right)\left(R^{-3\gamma_0+3}-1\right)+\frac{1}{We}\left(R^2-1\right)+\frac{Ca}{6}\left(R^3-1\right),$$
and 
$$E_0:=\frac{1}{2}\int_0^k u_0^2dx+\frac{Ca}{2}\frac{1}{\gamma-1}\int_0^k H(\rho_0)dx+P(R_0).$$
Then for any $t\in [0,T]$, 
\begin{equation}
\frac{1}{2}\int_0^ku^2dx+\frac{Ca}{2}\frac{1}{\gamma-1}\int_0^kH(\rho)dx+P(R)+\mu\int_0^t\int_0^k\rho(r^2u_x)^2dxd\tau
+2\mu\int_0^t\int_0^k\rho^{-1}\frac{u^2}{r^2}dxd\tau=E_0.
\end{equation}
\begin{proof}
Multiply (2.2) by $u$ to deduce that 
\begin{equation}
\frac{1}{2}\partial_t(u^2)+\frac{Ca}{2}((\rho^\gamma-1) r^2u)_x-\frac{Ca}{2}(\rho^\gamma-1)(r^2u)_x-\mu\left(\rho(r^2u)_xr^2u\right)_x+\mu\rho(r^2u)_x^2=0.
\end{equation}
(2.1) yields that $(\rho^\gamma-1)(r^2u)_x=-\frac{1}{\gamma-1}\partial_t H(\rho)$. From (2.5), it holds that $$\rho(r^2u)_x=\rho r^2u_x+2r^{-1}u,$$ $$(\rho(r^2)_xr^2u^2)_x=(2ru^2)_x=4\frac{u}{r}(r^2u_x)+2\rho^{-1}\frac{u^2}{r^2}.$$ Hence the cross term in $\rho(r^2u)_x^2$ is cancelled by the above boundary term. Then using boundary conditions (2.3)(2.4), the proof is complete by integrating (2.9) on $[0,k]\times[0,T]$ .
\end{proof}
\end{lem}
The most important part of the a-prior estimates is the control of both lower and upper bounds of $\rho$. This control is established through the Bresch-Desjardins entropy estimates stated in Lemma 2.5, which requires the control of $\|r^{-1}u\|_{L^\infty}$ and a boundary term involving $\rho|_{x=0}$. To this end, we first state the following two lemmas. In fact, using the dissipation terms in the basic energy identity and the radial property, a better $L^\infty$ control of $u$ can be proved:
\begin{lem}[$L^\infty$ control of $u$]
For any $t\in [0,T]$, $\int_0^t\|ur^{\frac{1}{2}}\|_{L^\infty}^2d\tau\leq\mu^{-1}E_0$.
\end{lem}
\begin{proof}
A direct computation shows that
$\partial_x(u^2r)=\rho^{-1}\frac{u^2}{r^2}+2\left(\rho^{-\frac{1}{2}}\frac{u}{r}\right)\left(\rho^{\frac{1}{2}} r^2u_x\right)$, and therefore $|\partial_x(u^2r)|\leq 2\rho^{-1}\frac{u^2}{r^2}+\rho (r^2u_x)^2$. Hence 
$$
\int_0^t\|ur^{\frac{1}{2}}\|_{L^\infty}^2d\tau\leq\int_0^t\int_0^k|\partial_x(u^2r)|dxd\tau\leq\int_0^t\int_0^k\rho(r^2u_x)^2dxd\tau+
2\int_0^t\int_0^k\rho^{-1}\frac{u^2}{r^2}dxd\tau\leq\mu^{-1}E_0.
$$
\end{proof}
For the control of $\rho|_{x=0}$, we have the following estimate. 
\begin{lem}[Control of $\rho|_{x=0}$]
There exists $0<c(T)<C(T)$, such that $c(T)\leq\rho|_{x=0}\leq C(T)$, for any $t\in [0,T]$.
\end{lem}
\begin{proof}
For simplicity, denote $\rho|_{x=0}$ by $\tilde{\rho}$. (2.1) gives that $$\rho r^2u_x=\rho(r^2u)_x-2r^{-1}u=-\rho^{-1}\partial_t\rho-2r^{-1}\partial_tr,$$ and thus $$(\rho r^2u_x)|_{x=0}=-\partial_t\left(\log(\tilde{\rho}R^2)\right)=\frac{1}{\gamma}(\tilde{\rho}R^2)^\gamma\partial_t\left((\tilde{\rho}R^2)^{-\gamma}\right).$$
Dividing (2.4) by $\mu(\tilde{\rho}R^2)^\gamma$ to deduce that 
\begin{equation}
\frac{d}{dt}\left((\tilde{\rho}R^2)^{-\gamma}\right)+\frac{\gamma}{\mu}\left[\left(\frac{Ca}{2}+\frac{2}{We}\right)R^{-3\gamma_0}-\frac{2}{We}R^{-1}\right](\tilde{\rho}R^2)^{-\gamma}=\frac{Ca}{2}\frac{\gamma}{\mu}R^{-2\gamma}.
\end{equation}
Solving (2.10) as an ODE of $(\tilde{\rho}R^2)^{-\gamma}$ yields that
\begin{equation}
(\tilde{\rho}R^2)^{-\gamma}(t)=(\tilde{\rho}R_0^2)^{-\gamma}S(t)+\frac{Ca}{2}\frac{\gamma}{\mu}\int_0^tR^{-2\gamma}S(t-\tau)d\tau,
\end{equation}
where $$S(t):=\exp\left\{-\frac{\gamma}{\mu}\int_0^t\left[\left(\frac{Ca}{2}+\frac{2}{We}\right)R^{-3\gamma_0}-\frac{2}{We}R^{-1}\right]d\tau\right\}.$$
Remark that $P(R)$ is convex for $R\in(0,+\infty)$, and reaches the minimum $0$ at $R=1$. Therefore, Lemma 2.2, which gives $P(R)\leq E_0$ implies that there exist $0<c<C$, such that for any $t\in [0,T]$,
\begin{equation} c\leq R(t)\leq C,\; r(x,t)=\left(R(t)^3+3\int_0^x\rho^{-1}(y,t)dy\right)^{\frac{1}{3}}\geq R(t)\geq c.
\end{equation}
Hence the proof is complete by (2.11).
\end{proof}
Using the above two lemmas, we are in a position to state the BD entropy estimate:
\begin{lem}[Bresch-Desjardins entropy estimate]
Define for $t\in[0,T]$ that
\begin{equation}
E_1(t):=\frac{1}{2}\int_0^k\left(u+\mu r^2(\log\rho)_x\right)^2dx+\frac{Ca}{2}\frac{1}{\gamma-1}\int_0^kH(\rho)dx+\frac{Ca}{2}\frac{4\mu}{\gamma}\int_0^t\int_0^k\left(r^2(\rho^{\frac{\gamma}{2}})_x\right)^2dxd\tau.
\end{equation}
There exists $C(T)>0$ such that for any $t\in[0,T]$,
$E_1(t)\leq C(T)$.
\end{lem}
\begin{proof}
Using (2.1), the viscous term can be rewritten as 
$$r^2(\rho(r^2u)_x)_x=-r^2(\log\rho)_{xt}=-\partial_t\left(r^2(\log\rho)_x\right)+2(\log\rho)_xru.$$ Take it into (2.2) to show $\partial_t\left(u+\mu r^2(\log\rho)_x\right)+\frac{Ca}{2}(\rho^\gamma)_xr^2=2\mu(\log\rho)_xru$. Then multiplying the resulted equation by $\left(u+\mu r^2(\log\rho)_x\right)$ and noting that 
$$(\rho^\gamma)_xr^2u=\left((\rho^\gamma-1)r^2u)\right)_x-(\rho^\gamma-1)(r^2u)_x=\left((\rho^\gamma-1)r^2u)\right)_x+\frac{1}{\gamma-1}\partial_t H(\rho),$$ 
it follows
\begin{equation}\begin{aligned}
&\frac{1}{2}\partial_t\left(u+\mu r^2(\log\rho)_x\right)^2+\frac{Ca}{2}\frac{1}{\gamma-1}\partial_t H(\rho)+\frac{Ca}{2}\frac{4\mu}{\gamma}\left(r^2(\rho^{\frac{\gamma}{2}})_x\right)^2\\=&2\mu(\log\rho)_xru(u+\mu(\log\rho)_xr^2)+\frac{Ca}{2}\left((1-\rho^\gamma)r^2u\right)_x.
\end{aligned}\end{equation}
Integrating (2.14) on $[0,k]$ yields that  
\begin{equation}\begin{aligned}
&\frac{1}{2}\frac{d}{dt}\int_0^k\left(u+\mu r^2(\log\rho)_x\right)^2dx+\frac{Ca}{2}\frac{1}{\gamma-1}\frac{d}{dt}\int_0^k H(\rho)dx+\frac{Ca}{2}\frac{4\mu}{\gamma}\int_0^k\left(r^2(\rho^{\frac{\gamma}{2}})_x\right)^2dx\\
=&2\mu\int_0^k(\log\rho)_xru(u+\mu(\log\rho)_xr^2)dx+\frac{Ca}{2}(\tilde{\rho}^\gamma-1)R^2\frac{dR}{dt}.
\end{aligned}\end{equation} 
Then controlling the two terms on the right-hand side of (2.15) using Lemma 2.3 and Lemma 2.4, we find
$$\begin{aligned}
\frac{Ca}{2}(\tilde{\rho}^\gamma-1)R^2\frac{dR}{dt}\leq C(T)\|ur^{\frac{1}{2}}\|_{L^\infty},
\end{aligned}$$
and
$$\begin{aligned}
&\mu\int_0^k(\log\rho)_xru(u+\mu(\log\rho)_xr^2)dx\\
\leq &2\left\|\frac{u}{r}\right\|_{L^\infty}\int_0^k\mu r^2(\log\rho)_x(u+\mu(\log\rho)_xr^2)dx\\
\leq&C\|ur^{\frac{1}{2}}\|_{L^\infty}\left[\int_0^k\left(u+\mu r^2(\log\rho)_x\right)^2dx+\int_0^ku^2dx\right],
\end{aligned}$$
which together with (2.15) complete the proof by Gronwall's inequality.
\end{proof}
Lemma 2.5 in fact provides a control for the $x$-derivative of $\rho$. Hence with the help of the radial property and Lemma 2.4, Lemma 2.2 and Lemma 2.5 give the lower and upper bounds of $\rho$: 
\begin{lem}[Lower and upper bound of density]
There exist $\underline{\rho}(T)>0$, $\overline{\rho}(T)>0$ such that for any $(x,t)\in [0,k]\times[0,T]$, $\underline{\rho}(T)\leq\rho(x,t)\leq\overline{\rho}(T)$. 
\end{lem}
\begin{proof}
Let $f_i(\alpha),\; i=1,2$ denote the two roots of $H(\rho)=\alpha$ . Let $\alpha\geq\frac{1}{k}\left(\frac{Ca}{2}\frac{1}{\gamma-1}\right)^{-1}E_0$, and thus $\alpha\geq\frac{1}{k}\int_0^kH(\rho)dx$. Since $\forall t\in [0,T]$,
$$\begin{aligned}
k>&m\left\{x\in(0,k):H(\rho)(x,t)>\alpha\right\}\\
=&m\left\{x\in(0,k):\rho(x,t)<f_1(\alpha)\right\}+m\left\{x\in(0,k):\rho(x,t)>f_2(\alpha)\right\},
\end{aligned}$$
there exists $x_0=x_0(t)\in[0,k]$ for each $t\in [0, T]$ such that $f_1(\alpha)\leq\rho(x_0(t),t)\leq f_2(\alpha)$. Then for any $(x,t)\in[0,k]\times[0,T]$, 
\begin{equation}
\left|\log\frac{\rho(x,t)}{\rho(x_0(t),t)}\right|
\leq\int_0^k|(\log\rho)_x|dx\leq\left(\int_0^k\left(r^2(\log\rho)_x\right)^2dx
\right)^{\frac{1}{2}}\left(\int_0^k r^{-4}dx\right)^{\frac{1}{2}}.
\end{equation}
To control the term $\int_0^kr^{-4}dx$, use the definition of $r$ (2.5) and (2.12) to calculate that
\begin{equation}
\frac{d}{dt}\int_0^kr^{-4}dx=-4\int_0^k r^{-5}udx\leq C\|ur^{\frac{1}{2}}\|_{L^\infty}\int_0^kr^{-4}dx.
\end{equation}
Applying Gronwall's inequality to (2.17) with the initial data $$\int_0^kr_0^{-4}dx\leq\int_0^k\left(R_0^3+3x\inf_{[0,k]}\rho_0^{-1}\right)^{-\frac{4}{3}}dx\leq R_0^{-1}\sup_{[0,k]}\rho_0$$ shows that $\int_0^kr^{-4}dx\leq C(T)$.
Therefore, in view of (2.8)(2.13)(2.16), there exist positive constant $C(T)$ such that $\left|\log\frac{\rho(x,t)}{\rho(x_0(t),t)}\right|\leq C(T)$, and thus the proof is complete by the selection of $\rho(x_0(t),t)$.
\end{proof}
To complete the a-priori estimates for generalized solution $(u,\;\rho,\; R)$, it remains to control the $L^\infty_tL^2_x$ norms of 1-order derivatives of $u$ and $\rho$, together with $r^2\partial_t\partial_x(\log\rho)$ and the higher order dissipation $r^2\partial_t\partial_x u$. Noting that $r^2u_x=\-\rho^{-2}\rho_t-\rho^{-1}\frac{u}{r}$ and $-\frac{Ca}{2}(\rho^\gamma)_x r^2=u_t+\mu r^2(\log\rho)_{xt}$, it suffices to establish the energy identity  and the BD entropy estimate of $(u_t,\;\rho_t,\; R_t)$.
\begin{lem}[Energy estimate for 1-order derivatives]
Define for $t\in[0,T]$ that 
$$\begin{aligned}
E_2(t):=&\frac{1}{2}\int_0^ku_t^2dx+\frac{Ca}{2}\int_0^k\left[\frac{2\gamma}{(\gamma-1)^2}\left(\rho^{\frac{\gamma-1}{2}}\right)_t^2+\frac{4}{\gamma-1}\rho^{\frac{\gamma-1}{2}}\left(\rho^{\frac{\gamma-1}{2}}\right)_t\frac{u}{r}+3\rho^{\gamma-1}\frac{u^2}{r^2}\right]dx\\
&+\frac{\mu}{2}\int_0^t\int_0^k\rho(r^2u_{tx})^2dxd\tau+\mu\int_0^t\int_0^k\rho^{-1}\frac{u_t^2}{r^2}dxd\tau\\
&+\left[\frac{3\gamma_0}{2}\left(\frac{Ca}{2}+\frac{2}{We}\right)R^{-3\gamma_0+1}+\frac{1}{We}\right]\left(\frac{dR}{dt}\right)^2.
\end{aligned}$$
There exists $C(T)>0$ such that $E_2(t)\leq C(T)$, $\forall t\in [0,T]$.
\end{lem}
\begin{proof}
To establish the energy identity for $(u_t,\;\rho_t,\;R_t)$, we differentiate (2.2) with respect to $t$, multiply the result equation by $u_t$ and compute each term.\\
\textbf{Step 1.} The treatment of 
$\left((\rho^\gamma)_xr^2\right)_tu_t$.\\
Exchanging the $x,t$ derivatives and applying integration by parts yield that
$$\begin{aligned}
\left((\rho^\gamma)_xr^2\right)_tu_t=
\left[(\rho^\gamma r^2)_tu_t\right]_x-(\rho^\gamma)_t(r^2u)_{xt}+(\rho^\gamma)_t\left((r^2)_tu\right)_x-\rho^\gamma\left((r^2)_tu_t\right)_x.
\end{aligned}$$
Then using (2.1), the second term can be rewritten as
$$-(\rho^\gamma)_t(r^2u)_{xt}=(\rho^\gamma)_t(\rho^{-2}\rho_t)_t=\frac{2\gamma}{(\gamma-1)^2}\partial_t\left[(\rho^{\frac{\gamma-1}{2}})_t\right]^2-\frac{\gamma(\gamma+1)}{2}\rho^{\gamma-4}(\partial_t\rho)^3.$$
Noting that $r_t=u$, exchanging the derivatives in the forth term gives
$$-\rho^\gamma\left((r^2)_tu_t\right)_x=-\rho^\gamma(r(u^2)_t)_x=-\left[\rho^\gamma(ru^2)_x\right]_t+\rho^\gamma(u^3)_x+(\rho^\gamma)_t(ru^2)_x.$$
Using (2.1) again, the first term on right-hand side is
$$-\left[\rho^\gamma(ru^2)_x\right]_t=\left[3\rho^{\gamma-1}\frac{u^2}{r^2}+\frac{4}{\gamma-1}\rho^{\frac{\gamma-1}{2}}(\rho^{\frac{\gamma-1}{2}})_t\frac{u}{r}\right]_t.$$
The rest nonlinear terms are
$$3(\rho^\gamma)_t(ru^2)_x=-3(\rho^\gamma)_t\left(2\rho^{-2}\rho_t\frac{u}{r}+3\rho^{-1}\frac{u^2}{r^2}\right)=\frac{-18\gamma}{\gamma-1}\rho^{\frac{\gamma-1}{2}}(\rho^{\frac{\gamma-1}{2}})_t\frac{u^2}{r^2}-\frac{24\gamma}{(\gamma-1)^2}(\rho^{\frac{\gamma-1}{2}})_t^2\frac{u}{r},$$
and
$$\rho^\gamma(u^3)_x=-\rho^\gamma(3\rho^{-2}\rho_t\frac{u^2}{r^2}+6\rho^{-1}\frac{u^3}{r^3})=-6\rho^{\gamma-1}\frac{u^3}{r^3}-\frac{6}{\gamma-1}\rho^{\frac{\gamma-1}{2}}(\rho^{\frac{\gamma-1}{2}})_t\frac{u^2}{r^2}.$$
Let $J$ collect all the nonlinear terms appeared, namely
$$J:=-\frac{\gamma(\gamma+1)}{2}\rho^{\gamma-4}(\partial_t\rho)^3-\frac{24\gamma}{(\gamma-1)^2}(\rho^{\frac{\gamma-1}{2}})_t^2\frac{u}{r}-\frac{6(3\gamma+1)}{\gamma-1}\rho^{\frac{\gamma-1}{2}}(\rho^{\frac{\gamma-1}{2}})_t\frac{u^2}{r^2}-6\rho^{\gamma-1}\frac{u^3}{r^3}.$$
As a result, 
\begin{equation}\begin{aligned}
\left((\rho^\gamma)_xr^2\right)_tu_t=\left[(\rho^\gamma r^2)_tu_t\right]_x+\frac{2\gamma}{(\gamma-1)^2}\partial_t(\rho^{\frac{\gamma-1}{2}})_t^2+\partial_t\left[3\rho^{\gamma-1}\frac{u^2}{r^2}+\frac{4}{\gamma-1}\rho^{\frac{\gamma-1}{2}}(\rho^{\frac{\gamma-1}{2}})_t\frac{u}{r}\right]+J.
\end{aligned}\end{equation}
\textbf{Step 2.} The treatment of $\left[(\rho(r^2u)_x)_xr^2\right]_tu_t$.\\
Exchanging $x,t$ derivatives and integrating by parts give
$$
\left[(\rho(r^2u)_x)_xr^2\right]_tu_t=
\left[(\rho(r^2u)_xr^2)_tu_t\right]_x-\rho(r^2u)_x((r^2)_tu_t)_x-(\rho(r^2u)_x)_t(r^2u_t)_x.$$
The boundary term can be rewritten as
$$\left[(\rho(r^2u)_xr^2)_tu_t\right]_x=\left[(\rho r^4u_x)_tu_t\right]_x+\left[(2ru)_tu_t\right]_x=\left[(\rho r^4u_x)_tu_t\right]_x+\left(2 u^2u_t\right)_x+(2ru_t^2)_x.$$
The third term, which involves the dissipation is
$$-(\rho(r^2u)_x)_t(r^2u_t)_x=-\rho(r^2u_t)_x^2+\rho^2(r^2u)_x^2(r^2u_t)_x+6\frac{u^2}{r^2}(r^2u_t)_x-4\rho\frac{u}{r}(r^2u)_x(r^2u_t)_x.$$
Using $r_t=u$, the second term is
$$-\rho(r^2u)_x((r^2)_tu_t)_x=6(r^2u)_x\frac{u}{r}\frac{u_t}{r}-2\rho(r^2u)_x^2\frac{u_t}{r}-2\rho(r^2u)_x\frac{u}{r}(r^2u_t)_x. $$
Let $K$ collect all the nonlinear terms, namely
$$K:=-6\rho\frac{u}{r}(r^2u)_x(r^2u_t)_x+\rho^2(r^2u)_x^2(r^2u_t)_x-2\rho(r^2u)_x^2\frac{u_t}{r}+6\frac{u^2}{r^2}(r^2u_t)_x+6(r^2u)_x\frac{u}{r}\frac{u_t}{r}.$$
Note that the cross term in $-\rho(r^2u_t)_x^2$ is cancelled by $(2ru_t^2)_x$. Hence we find
\begin{equation}\begin{aligned}
\left[(\rho(r^2u)_x)_xr^2\right]_tu_t=\left[(\rho r^4u_x)_tu_t\right]_x+\left(2 u^2u_t\right)_x-\rho(r^2u_{xt})^2-2\rho^{-1}\frac{u_t^2}{r^2}+K.
\end{aligned}\end{equation}
\textbf{Step 3.} The boundary term $\left.\left[\frac{Ca}{2}(\rho^\gamma r^2)_tu_t-\mu(\rho r^4u_x)_tu_t-2\mu u^2u_t\right]\right|_{x=0}$.\\
Differentiate the boundary condition (2.4) with respect to $t$.
\begin{equation}
\left.\left(\frac{Ca}{2}(\rho^\gamma r^2)_t-\mu\rho (r^4u_x)_t\right)\right|_{x=0}=\left[\left(\frac{Ca}{2}+\frac{2}{We}\right)R^{-3\gamma_0+2}-\frac{2}{We}R\right]_t.
\end{equation}
Multiplying (2.20) by $u_t|_{x=0}=\frac{d^2R}{dt^2}$, we obtain
\begin{equation}\begin{aligned}
&\left.\left[\frac{Ca}{2}(\rho^\gamma r^2)_tu_t-\mu(\rho r^4u_x)_tu_t\right]\right|_{x=0}\\
=&\left[\left(\frac{Ca}{2}+\frac{2}{We}\right)R^{-3\gamma_0+2}-\frac{2}{We}R\right]_t\frac{d^2R}{dt^2}\\
=&-\frac{3\gamma_0-2}{2}\left(\frac{Ca}{2}+\frac{2}{We}\right)R^{-3\gamma_0+1}\frac{d}{dt}\left(\frac{dR}{dt}\right)^2-\frac{1}{We}\frac{d}{dt}\left(\frac{dR}{dt}\right)^2\\
=&-\frac{3\gamma_0-2}{2}\left(\frac{Ca}{2}+\frac{2}{We}\right)\frac{d}{dt}\left[R^{-3\gamma_0+1}\left(\frac{dR}{dt}\right)^2\right]-\frac{1}{We}\frac{d}{dt}\left(\frac{dR}{dt}\right)^2\\
&-\frac{(3\gamma_0-2)(3\gamma_0-1)}{2}\left(\frac{Ca}{2}+\frac{2}{We}\right)R^{-3\gamma_0}\left(\frac{dR}{dt}\right)^3.
\end{aligned}\end{equation}
Let $L$ collect the nonlinear terms, namely
$$L=-\frac{(3\gamma_0-2)(3\gamma_0-1)}{2}\left(\frac{Ca}{2}+\frac{2}{We}\right)R^{-3\gamma_0}\left(\frac{dR}{dt}\right)^3-2\mu(u^2u_t)|_{x=0}.$$
Adding (2.18)(2.19) up with coefficient $\frac{Ca}{2}$ and $\mu$ respectively and integrating on $[0,k]$, one concludes with the help of (2.21) that
\begin{equation}\begin{aligned}
&\frac{1}{2}\frac{d}{dt}\int_0^ku_t^2dx+\frac{Ca}{2}\frac{d}{dt}\int_0^k\left[\frac{2\gamma}{(\gamma-1)^2}(\rho^{\frac{\gamma-1}{2}})_t^2+\frac{4}{\gamma-1}\rho^{\frac{\gamma-1}{2}}(\rho^{\frac{\gamma-1}{2}})_t\frac{u}{r}+3\rho^{\gamma-1}\frac{u^2}{r^2}\right]dx\\
&+\frac{d}{dt}\left[\frac{3\gamma_0-2}{2}\left(\frac{Ca}{2}+\frac{2}{We}\right)R^{-3\gamma_0+1}\left(\frac{dR}{dt}\right)^2+\frac{1}{We}\left(\frac{dR}{dt}\right)^2\right]\\
&+\mu\int_0^k\rho(r^2u_{tx})^2dx+2\mu\int_0^k\rho^{-1}\frac{u_t^2}{r^2}dx\\
=&-\frac{Ca}{2}\int_0^kJdx+\mu\int_0^k K dx+L.
\end{aligned}\end{equation}
\textbf{Step 4.} Control of the nonlinear terms.\\
First note that by (2.2),
\begin{equation}\begin{aligned}
\int_0^x\frac{u_t}{r^2}dy
=&\left(\mu\rho(r^2u)_x-\frac{Ca}{2}\rho^\gamma\right)(x,t)-\left(\mu\rho(r^2u)_x-\frac{Ca}{2}\rho^\gamma\right)(0,t)\\
=&\left(\mu\rho(r^2u)_x-\frac{Ca}{2}\rho^\gamma\right)(x,t)+\left(\frac{Ca}{2}+\frac{2}{We}\right)R^{-3\gamma_0}-\frac{2}{We}R^{-1}-2\mu R^{-1}\frac{dR}{dt}.
\end{aligned}\end{equation}
In view of equation (2.1), there is $\rho(r^2u)_x=-\rho^{-1}\partial_t\rho$. Hence replacing $\rho(r^2u)_x$ by $-\rho^{-1}\partial_t\rho$ in (2.23) yields that
\begin{equation}\begin{aligned}
\|\mu\rho^{-1}\rho_t\|_{L^\infty}
\leq &\left\|\frac{Ca}{2}\rho^\gamma-\left(\frac{Ca}{2}+\frac{2}{We}\right)R^{-3\gamma_0}+\frac{2}{We}R^{-1}\right\|_{L^\infty}+2\mu R^{-1}\left|\frac{dR}{dt}\right|+\int_0^k\left|\frac{u_t}{r^2}\right|dx\\
\leq&C(T)+2\mu R^{-\frac{3}{2}}\|ur^{\frac{1}{2}}\|_{L^\infty}+\left(\int_0^ku_t^2dx\right)^\frac{1}{2}\left(\int_0^kr^{-4}dx\right)^\frac{1}{2}\\
\leq& C(T)\left(1+\left(\int_0^k u_t^2dx\right)^{\frac{1}{2}}+\|ur^{\frac{1}{2}}\|_{L^\infty}\right),
\end{aligned}\end{equation} 
where we used the boundedness of $\rho$ and $R$ from Lemma 2.6 and (2.12).
Using (2.24), the four terms in $J$ can be controlled as following.
$$\begin{aligned}
\left|\int_0^k\rho^{\gamma-4}\rho_t^3dx\right|
\leq &C\|\rho^{-1}\rho_t\|_{L^\infty}\int_0^k(\rho^{\frac{\gamma-1}{2}})_t^2dx\\
\leq& C(T)\left(1+\left(\int_0^k u_t^2dx\right)^{\frac{1}{2}}+\|ur^{\frac{1}{2}}\|_{L^\infty}\right)\int_0^k(\rho^{\frac{\gamma-1}{2}})_t^2dx .\end{aligned}$$
$$\left|\int_0^k(\rho^{\frac{\gamma-1}{2}})_t^2\frac{u}{r}dx\right|\leq\left\|\frac{u}{r}\right\|_{L^\infty}\int_0^k(\rho^{\frac{\gamma-1}{2}})_t^2dx\leq C\|ur^{\frac{1}{2}}\|_{L^\infty}\int_0^k(\rho^{\frac{\gamma-1}{2}})_t^2dx.$$
$$\left|\int_0^k\rho^{\frac{\gamma-1}{2}}(\rho^{\frac{\gamma-1}{2}})_t\frac{u^2}{r^2}dx\right|\leq C\|ur^{\frac{1}{2}}\|_{L^\infty}\left(\int_0^k(\rho^{\frac{\gamma-1}{2}})_t^2dx+\int_0^k\rho^{\gamma-1}\frac{u^2}{r^2} dx\right).$$
$$\left|\int_0^k\rho^{\gamma-1}\frac{u^3}{r^3}dx\right|\leq C\|ur^{\frac{1}{2}}\|_{L^\infty}\int_0^k\rho^{\gamma-1}\frac{u^2}{r^2}dx.$$
Adding the above inequalities up gives the control of $\int Jdx$:
\begin{equation}
\left|\int_0^kJdx\right|\leq C(T)\left(1+\left(\int_0^k(\rho^{\frac{\gamma-1}{2}})_t^2dx\right)^{\frac{1}{2}}+\|ur^{\frac{1}{2}}\|_{L^\infty}\right)\left(\int_0^ku_t^2dx+\int_0^k(\rho^{\frac{\gamma-1}{2}})_t^2+\int_0^k\rho^{\gamma-1}\frac{u^2}{r^2}\right).
\end{equation}
Let $\epsilon>0$ be a small constant. Using Lemma 2.6, inequality (2.24) and the equation (2.1), the terms in $K$ have the following controls.
$$\begin{aligned}
\int_0^k\rho\frac{u}{r}(r^2u)_x(r^2u_t)dx
\leq &C\left\|\frac{u}{r}\right\|_{L^\infty}\left(\int_0^k\rho(r^2u_t)_x^2dx\right)^{\frac{1}{2}}\left(\int_0^k\rho(r^2u)_x^2dx\right)^{\frac{1}{2}}\\
\leq&\epsilon\int_0^k\rho(r^2u_t)_x^2dx+C_\epsilon(T)\|ur^{\frac{1}{2}}\|_{L^\infty}^2\int_0^k(\rho^{\frac{\gamma-1}{2}})_t^2dx.
\end{aligned}$$
$$\begin{aligned}
\int_0^k\rho^2(r^2u)_x^2(r^2u_t)_xdx
\leq&\left(\int_0^k\rho(r^2u_t)_x^2dx\right)^{\frac{1}{2}}\left(\int_0^k\rho^3(r^2u)_x^4dx\right)^{\frac{1}{2}}\\
\leq&\epsilon\int_0^k\rho(r^2u_t)_x^2dx
+C_\epsilon\int_0^k\rho^3(r^2u)_x^4dx\\
\leq&\epsilon\int_0^k\rho(r^2u_t)_x^2dx+C_\epsilon(T)\int_0^k(\rho^{\frac{\gamma-1}{2}})_t^2dx\left(1+\int_0^ku_t^2dx+\|ur^{\frac{1}{2}}\|_{L^\infty}^2\right).
\end{aligned}$$
$$\int_0^k\rho(r^2u)_x^2\frac{u_t}{r}dx\leq\epsilon\int_0^k\rho^{-1}\frac{u_t^2}{r^2}dx+C_\epsilon(T)\int_0^k(\rho^{\frac{\gamma-1}{2}})_t^2dx\left(1+\int_0^ku_t^2dx+\|ur^{\frac{1}{2}}\|_{L^\infty}^2\right).$$
$$\int_0^k(r^2u)_x\frac{u}{r}\frac{u_t}{r}dx\leq\epsilon\int_0^k\rho^{-1}\frac{u_t^2}{r^2}dx+C_\epsilon(T)\|ur^{\frac{1}{2}}\|_{L^\infty}^2\int_0^k(\rho^{\frac{\gamma-1}{2}})_t^2dx.$$
$$\int_0^k(r^2u_t)_x\frac{u^2}{r^2}dx\leq\epsilon\int_0^k\rho(r^2u_t)_x^2dx+C_\epsilon(T)\|ur^{\frac{1}{2}}\|_{L^\infty}^2\int_0^k\rho^{\gamma-1}\frac{u^2}{r^2}dx.$$
Adding the above estimates up and noting that $$\int_0^k\rho(r^2u_t)_x^2dx\leq C\int_0^k\rho(r^2u_{xt})^2dx+C\int_0^k\rho^{-1}\frac{u_t^2}{r^2}dx,$$
one gets the control of $\int Kdx$:
\begin{equation}\begin{aligned}
&\left|\int_0^kKdx\right|\\
\leq&C_\epsilon(T)\left(1+\int_0^k(\rho^{\frac{\gamma-1}{2}})_t^2dx+\|ur^{\frac{1}{2}}\|_{L^\infty}^2\right)\left(\int_0^ku_t^2dx+\int_0^k(\rho^{\frac{\gamma-1}{2}})_t^2dx+\int_0^k\rho^{\gamma-1}\frac{u^2}{r^2}dx\right)\\
&+\epsilon\int_0^k\rho(r^2u_{xt})^2dx+\epsilon\int_0^k\rho^{-1}\frac{u_t^2}{r^2}dx.
\end{aligned}\end{equation}
At last, to estimate $L$, we use $\partial_x(u_t^2r)=2(\rho^{-\frac{1}{2}}\frac{u_t}{r})(\rho^\frac{1}{2}r^2u_{xt})+\rho^{-1}\frac{u_t^2}{r^2}$ to obtain
$$\left|R^{-3\gamma_0}(\frac{dR}{dt})^3\right|\leq C\|ur^{\frac{1}{2}}\|_{L^\infty}\left(\frac{dR}{dt}\right)^2,$$
and
$$|(u^2u_t)|_{x=0}|
\leq\epsilon\|u_tr^{\frac{1}{2}}\|^2_{L^\infty}
+C_\epsilon(u|_{x=0})^4\leq\epsilon\int_0^k\rho(r^2u_{tx})^2dx
+2\epsilon\int_0^k\rho^{-1}\frac{u_t^2}{r^2}dx+C_\epsilon\|ur^{\frac{1}{2}}\|_{L^\infty}^2\left(\frac{dR}{dt}\right)^2.$$
Now use the above two inequalities and (2.25)(2.26) in (2.22) and choose $\epsilon$ small enough to deduce that
\begin{equation}
\frac{dE_2}{dt}\leq C(T)\left(1+\|ur^{\frac{1}{2}}\|_{L^\infty}^2+\int_0^k(\rho^{\frac{\gamma-1}{2}})_t^2dx\right)E_2.
\end{equation}
Note that $$(\rho^{\frac{\gamma-1}{2}})_t^2=C\rho^{\gamma-3}\rho_t^2=C\rho^{\gamma+1}(r^2u)_x^2\leq C(T)\left(\rho(r^2u_x)^2+2\rho^{-1}\frac{u^2}{r^2}\right),$$ 
which is integrable in time by Lemma 2.2. Hence using Gronwall's inequality to (2.27) in view of Lemma 2.2 and Lemma 2.3, if follows that $E_2(t)\leq C(T)$, $\forall t\in[0,T].$
\end{proof}
\begin{lem}[Bresch-Desjardins entropy estimate for 1-order derivatives]
Define for $t\in[0,T]$ that
$$\begin{aligned}
E_3(t):=&\frac{1}{2}\int_0^k\left(u_t+\mu r^2(\log\rho)_{xt}\right)^2dx+\frac{Ca}{4}\mu\gamma\int_0^t\int_0^k\rho^\gamma(\log\rho)_{xt}^2dxd\tau
\\&+\frac{Ca}{2}\int_0^k\left[\frac{2\gamma}{(\gamma-1)^2}\left(\rho^{\frac{\gamma-1}{2}}\right)_t^2+\frac{4}{\gamma-1}\rho^{\frac{\gamma-1}{2}}\left(\rho^{\frac{\gamma-1}{2}}\right)_t\frac{u}{r}+3\rho^{\gamma-1}\frac{u^2}{r^2}\right]dx.
\end{aligned}$$
There exists $C(T)>0$ such that $E_3(t)\leq C(T),\;\forall t\in [0,T]$.
\end{lem}
\begin{proof}
By differentiating (2.2) with respect to $t$ and using $(\rho(r^2u)_x)_x=(\log\rho)_{xt}$, it holds that
\begin{equation}
\partial_t\left(u_t+\mu r^2(\log\rho)_{xt}\right)+\frac{Ca}{2}\partial_t\left((\rho^\gamma)_xr^2\right)=0.
\end{equation}
Multiply (2.28) by $(u_t+\mu r^2(\log\rho)_{xt})$ to deduce that
\begin{equation}
\frac{1}{2}\partial_t\left(u_t+\mu r^2(\log\rho)_xt\right)^2+\frac{Ca}{2}\left((\rho^\gamma)_xr^2\right)_tu_t+\frac{Ca}{2}\mu\left((\rho^\gamma)_xr^2\right)_t(\log\rho)_{xt}r^2=0.
\end{equation}
Treat the second term by the same way as in  \textbf{Step 1} of Lemma 2.7, and then (2.29) yields that
\begin{equation}\begin{aligned}
&\frac{1}{2}\partial_t\left(u_t+\mu r^2(\log\rho)_xt\right)^2+\frac{Ca}{2}\partial_t\left[\frac{2\gamma}{(\gamma-1)^2}(\rho^{\frac{\gamma-1}{2}})_t^2+3\rho^{\gamma-1}\frac{u^2}{r^2}+\frac{4}{\gamma-1}\rho^{\frac{\gamma-1}{2}}(\rho^{\frac{\gamma-1}{2}})_t\frac{u}{r}\right]\\
&+\frac{Ca}{2}\left[(\rho^\gamma r^2)_tu_t\right]_x+\frac{Ca}{2}\mu\gamma(\rho^\gamma)_x(\log\rho)_t(\log\rho)_{xt}r^4+2\frac{Ca}{2}\mu(\rho^\gamma)_x(\log\rho)_{xt}r^3u\\
&+\frac{Ca}{2}\mu\gamma\rho^\gamma r^4(\log\rho)_{xt}^2+\frac{Ca}{2}J=0.
\end{aligned}\end{equation}
Estimate $\left|\int_0^k Jdx\right|$ in the same way as in \textbf{Step 4} in the proof of Lemma 2.7, and note that Lemma 2.7 already shows that $\int_0^k u_t^2dx\leq C(T)$. Then one gets the control of the $J$ terms:
\begin{equation}
\left|\int_0^kJdx\right|\leq C(T)(1+\|ur^{\frac{1}{2}}\|_{L^\infty})\left(\int_0^k(\rho^{\frac{\gamma-1}{2}})_t^2dx+\int_0^k\rho^{\gamma-1}\frac{u^2}{r^2}dx\right).
\end{equation}
Note that $-\frac{Ca}{2}(\rho^\gamma)_x r^2=\partial_t u+\mu r^2(\log\rho)_{xt}$. Using (2.24), the rest two nonlinear terms are estimated as following:
\begin{equation}\begin{aligned}
&\left|\int_0^k(\rho^\gamma)_x(\log\rho)_t(\log\rho)_{xt}r^4dx\right|\\
\leq&\epsilon\int_0^k\rho^\gamma r^4(\log\rho)_{xt}^2dx+C_\epsilon\int_0^k\rho^{-\gamma}(\rho^\gamma)_x^2(\log\rho)_t^2r^4dx\\
\leq&\epsilon\int_0^k\rho^\gamma r^4(\log\rho)_{xt}^2dx
+C_\epsilon(T)(1+\|ur^{\frac{1}{2}}\|_{L^\infty}^2)\int_0^k(u_t+\mu r^2(\log\rho)_{xt})^2dx,
\end{aligned}\end{equation}
\begin{equation}\begin{aligned}
&\left|\int_0^k(\rho^\gamma)_x(\log\rho)_{xt}r^3udx\right|\\
\leq&\epsilon\int_0^k\rho^\gamma r^4(\log\rho)_{xt}^2dx+C_\epsilon\int_0^k\rho^{-\gamma}(\rho^\gamma)_x^2r^4\frac{u^2}{r^2}dx\\
\leq&\epsilon\int_0^k\rho^\gamma r^4(\log\rho)_{xt}^2dx+C_\epsilon(T)\|ur^{\frac{1}{2}}\|_{L^\infty}^2\int_0^k(u_t+\mu r^2(\log\rho)_{xt})^2dx.
\end{aligned}\end{equation}
To control the boundary term $\left[(\rho^\gamma r^2)_tu_t\right]|_{x=0}$, first by Lemma 2.4 and (2.10), there exists $C(T)>0$ such that $\left|\frac{d}{dt}(\tilde{\rho}R^2)^\gamma\right|\leq C(T)$. Then using Lemma 2.7, it follows that 
\begin{equation}
|(\rho^\gamma r^2)_t|_{x=0}|=\left|\frac{d}{dt}(\tilde{\rho}R^2)^\gamma R^{2-2\gamma}+(\tilde{\rho}R^2)^\gamma\frac{d}{dt}R^{2-2\gamma}\right|\leq C(T).
\end{equation}
Again by Lemma 2.7, 
\begin{equation}
\int_0^t |u_t|_{x=0}|^2d\tau\leq\int_0^t\|u_tr^{\frac{1}{2}}\|_{L^\infty}dx\leq\int_0^t\int_0^k\rho(r^2u_{tx})^2dxd\tau
+2\int_0^t\int_0^k\rho^{-1}\frac{u_t^2}{r^2}dxd\tau\leq C(T).
\end{equation}
Finally, by choosing $\epsilon$ small enough, integrating (2.30) on $[0,k]$, and use Gronwall's inequality with the help of (2.31-2.35), the proof is complete.
\end{proof}
\noindent\textbf{Proof of Proposition 2.1}.
Proposition 2.1 is proved through short time existence, a-priori estimates in Lemma 2.2-2.8 and a continuity argument. The short time existence under the assumptions of Proposition 2.1 can be shown by using energy estimates and Galerkin approximation as in \cite[chapter 2]{BVPMNF}, see also \cite{GWS1D}. The equations $$r^2\partial_x(\log\rho)=(\gamma\rho^\gamma)^{-1}r^2\partial_x(\rho^\gamma)=(\frac{Ca}{2}\gamma\rho^\gamma)^{-1}(u_t+\mu r^2(\log\rho)_{xt})$$ and $r^2\partial_xu=\partial_x(r^2u)-\frac{u}{\rho r}=-\rho^{-2}\partial_t\rho-\frac{u}{\rho r}$ show that $\|r^2\partial_x(\log\rho)\|_{L^2}$ and $\|r^2\partial_xu\|_{L^2}$ are controlled by the bounds given by Lemma 2.7 and Lemma 2.8 with  coefficients depending on $\sup_{x\in[0,k]}\rho$ and $\left(\inf_{x\in[0,k]}\rho\right)^{-1}$, which are also bounded by Lemma 2.6. Therefore the global existence of generalized solution to (2.1-2.7) is proved by using a standard continuity argument with the a-priori estimates in Lemma 2.2-2.8.
\section{Construction of the global solution and the Uniqueness}
This section is devoted to the construction and its uniqueness of global generalized solutions. The construction is under the same frame as in \cite{GSS}, in which solutions on bounded domains are regarded as approximate solutions,  and then compactness argument is applied to obtain the wanted solution to the original problem on the unbounded exterior domain. Since the problem considered in this paper involves an additional free boundary compared with \cite{GSS}, for the sake of rigorousness, we give in this section an explicit description to the construction.
\subsection{Construction of the approximate solutions}
Let $\phi$ be a smooth cut off function on $\mathbb{R}^+$ such that $0\leq\phi\leq 1$, $\phi(z)=1$ for $z\in [0,\frac{1}{2}]$, $\phi(z)=0$ for $z\geq 1$, $\left|\frac{d^i\phi}{dz^i}\right|\leq C$ for $i=1,2,3$ and $z\in\mathbb{R}^+$. Define $\phi_k(z)=\phi(\frac{z}{k})$ for $k\in\mathbb{N}$. Now for the initial value $(u_0,\;\rho_0,\;R_0)$ satisfying the assumptions in Theorem 1.2, define for $k\in\mathbb{N}$ that
\begin{equation}
u_{k,0}:=u_0\phi_k,\;\rho_{k,0}^{-1}:=1+(\rho_0^{-1}-1)\phi_k,\; R_{k,0}:=R_0.
\end{equation}
Similarly, define $$r_{k,0}(x)=\left(R_{k,0}^3+3\int_0^x\rho_{k,0}^{-1}(x)dy\right)^{\frac{1}{3}}.$$ Noting that $\rho_0$ is bounded from both above and below, it is easy to check that 
\begin{equation}\begin{aligned}
&\left(u_{k,0},\;\rho_{k,0},\;r_{k,0}^2(u_{k,0})_x,\;r^2_{k,0}(\rho_{k,0})_x,\;r_{k,0}^2(\rho_{k,0}(r^2_{k,0}u_{k,0})_x)_x\right)\\ 
\rightarrow &\left(u_0,\;\rho_0,\;r_0^2(u_0)_x,\;r_0^2(\rho_0)_x,\;r_0^2(\rho_0(r_0^2u_0)_x)_x\right)\text{ in }L^2,\text{ as }k\rightarrow+\infty.
\end{aligned}\end{equation}
\indent Now let $(u_k,\;\rho_k,\;R_k)$ be given by Proposition 2.1 with initial data $(u_{k,0},\;\rho_{k,0},\;R_{k,0})$. Define $r_k$ as in (2.5) with $(u,\;\rho,\;R)$ replaced by $(u_k,\;\rho_k,\;R_k)$. Then the estimates established by Lemma 2.2-2.8 hold for each $(u_k,\;\rho_k,\; R_k)$ with the initial data $(u_{k,0},\;\rho_{k,0},\; R_{k,0})$.  Now define $\tilde{u}_k=u_k\phi_k$ and $\tilde{\rho}_k^{-1}=1+(\rho_k^{-1}-1)\phi_k$ for $k\in\mathbb{N}.$ Let $T>0$ be arbitrary. Then by definition, for $k>2N$
\begin{equation}
(\tilde{u}_k,\;\tilde{\rho}_k)=(u_k,\;\rho_k),\;\forall (x,t)\in[0,N]\times[0,T].  
\end{equation}
Denote $Q_T:=[0,+\infty)\times [0,T]$, $Q_{k,T}:=[0,k]\times[0,T]$, and abbreviate $\|\cdot\|_{L^p_tL^q_x(Q_T)}$ as $\|\cdot\|_{L^p_tL^q_x}$. Remark again that the estimates in Lemma 2.2-2.8 are not dependent on $k$ but only on the norms of the initial value, which is by (3.1) and (3.2) uniformly bounded in $k$. We then check by applying Lemma 2.2-2.8 on $(u_k,\;\rho_k,\;R_k)$ and (3.2) that
$$
\|\tilde{u}_k\|_{L^\infty_tL^2_x}=\|u_k\phi_k\|_{L^\infty_tL^2_x}\leq\|u_k\|_{L^\infty_tL^2_x(Q_{k,T})}\leq C,$$
\begin{equation}
\|\partial_t\tilde{u}_k\|_{L^\infty_tL^2_x}=\|\partial_tu_k\phi_k\|_{L^\infty_tL^2_x}\leq\|\partial_t u_k\|_{L^\infty_tL^2_x(Q_{k,T})}\leq C(T),
\end{equation}
\begin{equation}
\inf_{(x,t)\in Q_T}\tilde{\rho}_k\geq\inf_{(x,t)\in Q_{k,T}}\rho_k\geq c(T),\;\sup_{(x,t)\in Q_T}\tilde{\rho}_k\leq\sup_{(x,t)\in Q_{k,T}}\rho_k\leq C(T),
\end{equation}
\begin{equation}
\int_0^\infty H(\tilde{\rho}_k)dx\leq C(T)\int_0^\infty(\tilde{\rho}_k^{-1}-1)^2dx\leq C(T)\int_0^k(\rho_k^{-1}-1)^2dx\leq C(T)\int_0^k H(\rho_k)dx\leq C(T),
\end{equation}
$$\|\partial_t\tilde{\rho_k}^{-1}\|_{L^\infty_tL^2_x}=\|\partial_t\rho_k^{-1}\phi_k\|_{L^\infty_tL^2_x}\leq\|\partial_t\rho_k^{-1}\|_{L^\infty_tL^2_x(Q_{k,T})}\leq C(T).$$
To bound the norms involving $x$-derivatives, first by the definition (2.5) of $r_k$, it holds that
\begin{equation}
c(T)(1+3x)\leq r_k^3\leq C(T)(1+3x).
\end{equation} 
We then control the norm of $x$-derivative of $\tilde{u}_k$ by
\begin{equation}\begin{aligned}
\|(1+3x)^{\frac{2}{3}}(\tilde{u}_k)_x\|_{L^\infty_tL^2_x}
\leq&\|(1+3x)^{\frac{2}{3}}(u_k)_x\phi_k\|_{L^\infty_tL^2_x}+\|(1+3x)^{\frac{2}{3}}(\phi_k)_xu_k\|_{L^\infty_tL^2_x}\\
\leq&C(T)\|r_k^2(u_k)_x\|_{L^\infty_tL^2_x(Q_{k,T})}+\|(1+3x)^{\frac{2}{3}}(\phi_k)_xu_k\|_{L^\infty_tL^2_x}\\
\leq& C(T),
\end{aligned}\end{equation}
where in the last step we use the inequality $$\left|(1+3x)^{\frac{2}{3}}(\phi_k)_x\right|\leq C\chi_{\left\{2^{-1}k\leq x\leq k\right\}}k^{-1}(1+3x)^{\frac{2}{3}}\leq C(1+3x)^{-\frac{1}{3}}. $$ 
Similarly, the $x$-derivative of $\tilde{\rho}_k$ have the control that
\begin{equation}\begin{aligned}
\|(1+3x)^{\frac{2}{3}}(\tilde{\rho}_k^{-1})_x\|_{L^\infty_tL^2_x}
\leq&\|(1+3x)^{\frac{2}{3}}(\rho_k^{-1})_x\phi_k\|_{L^\infty_tL^2_x}+\|(1+3x)^{\frac{2}{3}}(\phi_k)_x(\rho_k^{-1}-1)\|_{L^\infty_tL^2_x}\\
\leq&C(T)\|r_k^2(\rho_k^{-1})_x\|_{L^\infty_tL^2_x(Q_{k,T})}+\|(1+3x)^{\frac{2}{3}}(\phi_k)_x(\rho_k^{-1}-1)\|_{L^\infty_tL^2_x}
\\
\leq &C(T).
\end{aligned}\end{equation}
The mixed derivative of $u$ can be bounded easily by using (3.7) that
\begin{equation}\begin{aligned}
\|(1+3x)^{\frac{2}{3}}(\tilde{u}_k)_{xt}\|_{L^2_tL^2_x}
\leq&\|(1+3x)^{\frac{2}{3}}(u_k)_{xt}\phi_k\|_{L^2_tL^2_x}+\|(1+3x)^{\frac{2}{3}}(\phi_k)_x(u_k)_t\|_{L^2_tL^2_x}\\
\leq& C(T)\|r_k^2(u_k)_{xt}\|_{L^2_tL^2_x(Q_{k,T})}+C(T)\|\frac{(u_k)_t}{r_k}\|_{L^2_tL^2_x(Q_{k,T})}\\
\leq &C(T).
\end{aligned}\end{equation}
To control the mixed derivative of $\rho_k$, first note that (2.24) and the inequality 
$$\begin{aligned}
\|u_k^2r\|_{L^\infty_x(0,k)}\leq&\int_0^k\rho_k(r_k^2 (u_k)_x)^2dx+2\int_0^k\rho_k^{-1}\frac{u_k^2}{r_k^2}dx\\
=&\int_0^k\rho_k(r_k^2u_k)_x^2dx+2(r_k u_k^2)|_{x=0}\\
\leq& C(T)\left(\|\partial_t\rho_k^{-1}\|_{L^2_x(0,k)}^2+\left(\frac{dR_k}{dt}\right)^2\right)
\end{aligned}$$
imply that $\|(\log\rho_k)_t\|_{L^\infty_tL^\infty_x(Q_{k,T})}\leq C(T)$. Therefore
\begin{equation}\begin{aligned}
&\|(1+3x)^{\frac{2}{3}}(\log\tilde{\rho}_k)_{xt}\|_{L^\infty_tL^2_x}\\
\leq &C(T)\|r_k^2\left(\tilde{\rho}_k\rho_k^{-1}(\log\rho_k)_t\phi_k\right)_x\|_{L^\infty_tL^2_x}\\\leq& C(T)\|r_k^2(\tilde{\rho}_k)_x\rho_k^{-1}(\log\rho_k)_t\phi_k\|_{L^\infty_tL^2_x}+C(T)\|r_k^2\tilde{\rho}_k(\rho_k^{-1})_x(\log\rho_k)_t\phi_k\|_{L^\infty_tL^2_x}\\
&+C(T)\|r_k^2\tilde{\rho}_k\rho_k^{-1}(\log\rho_k)_{xt}\phi_k\|_{L^\infty_tL^2_x}+C(T)\|r_k^2\tilde{\rho}_k\rho_k^{-1}(\log\rho_k)_t(\phi_k)_x\|_{L^\infty_tL^2_x}\\
\leq& C(T)\|(\log\rho_k)_t\|_{L^\infty_tL^\infty_x(Q_{k,T})}\left(\|r_k^2(\rho_k^{-1})_x\|_{L^\infty_tL^2_x(Q_{k,T})}+\|r_k^2(\tilde{\rho}_k^{-1})_x\|_{L^\infty_tL^2_x}\right)\\
&+C(T)\|r_k^2(\log\rho_k)_{xt}\|_{L^\infty_t L^2_x(Q_{k,T})}+C(T)\|r_k^2(\phi_k)_x\partial_t\rho_k^{-1}\|_{L^\infty_tL^2_x}\\
\leq &C(T).
\end{aligned}\end{equation}
Summarizing the estimates (3.4-3.11), one concludes that
\begin{equation}\begin{aligned}
&\|\tilde{u}_k,\;(\tilde{\rho}^{-1}_k-1),\;\partial_t\tilde{u}_k,\;\partial_t\tilde{\rho}^{-1},\;(1+3x)^{\frac{2}{3}}\partial_x\tilde{u}_k,\;(1+3x)^{\frac{2}{3}}\partial_x\tilde{\rho}^{-1}_k,\;(1+3x)^{\frac{2}{3}}(\log\tilde{\rho})_{xt}\|_{L^\infty_tL^2_x}^2\\
&+\int_0^T\|(1+3x)^{\frac{2}{3}}(\tilde{u}_k)_{xt}\|_{L^2}^2d\tau\leq C(T).
\end{aligned}\end{equation}
Hence there exist functions $(u,\;\rho^{-1})$ and a subsequence of $(\tilde{u}_k,\;\tilde{\rho}^{-1}_k)$ (still denoted by $(\tilde{u}_k,\;\tilde{\rho}^{-1}_k)$) such that as $k\rightarrow+\infty$,
\begin{equation}\begin{aligned}
&\left(\tilde{u}_k,\;(\tilde{\rho}^{-1}-1),\;\partial_t\tilde{u}_k,\;\partial_t\tilde{\rho}^{-1},\;(1+3x)^{\frac{2}{3}}\partial_x\tilde{u}_k,\;(1+3x)^{\frac{2}{3}}\partial_x\tilde{\rho}^{-1}_k,\;(1+3x)^{\frac{2}{3}}(\log\tilde{\rho})_{xt}\right)\\
\rightharpoonup &\left(u,\;(\rho^{-1}-1),\;\partial_tu,\;\partial_t\rho^{-1},\;(1+3x)^{\frac{2}{3}}\partial_xu,\;(1+3x)^{\frac{2}{3}}\partial_x\rho^{-1},\;(1+3x)^{\frac{2}{3}}(\log\rho)_{xt}\right)\\
&\text{in the weak-$\ast$ sense of $L^\infty([0,T],L^2)$, and that}  \\
&(1+3x)^{\frac{2}{3}}(u_k)_{xt}\rightharpoonup (1+3x)^{\frac{2}{3}}u_{xt}\text{ in the weak sense of $L^2\left([0,T],L^2\right)$,}
\end{aligned}\end{equation}
with $(u,\;\rho)$ satisfying 
\begin{equation}\begin{aligned}
&\|u,\;(\rho^{-1}-1),\;\partial_tu,\;\partial_t\rho^{-1},\;(1+3x)^{\frac{2}{3}}\partial_xu,\;(1+3x)^{\frac{2}{3}}\partial_x\rho^{-1}_k,\;(1+3x)^{\frac{2}{3}}(\log\rho)_{xt}\|_{L^\infty_tL^2_x}^2\\
&+\int_0^T\|(1+3x)^{\frac{2}{3}}u_{xt}\|_{L^2}^2d\tau\leq C(T).
\end{aligned}\end{equation}
Moreover, for any $\psi\in C_c^\infty(Q_T)$ with $\psi\geq0$, since
$\lim_{k\rightarrow+\infty}\int_{Q_T}\tilde{\rho}_k\psi dxdt=\int_{Q_T}\rho\psi dxdt,$ and
$c(T)\int_{Q_T}\psi dxdt\leq\lim_{k\rightarrow+\infty}\int_{Q_T}\tilde{\rho}_k\psi dxdt\leq C(T)\int_{Q_T}\psi dxdt,$ 
it holds that \begin{equation}
c(T)\leq\rho\leq C(T), \text{ on } Q_T.
\end{equation}
\indent Define $r(x,t)=r_0(x)+\int_0^tud\tau,\;R(t)=r(0,t).$
 We now check that $(u,\;\rho,\;R)$ is a generalized solution to (1.14-1.19). First, (1.16) holds immediately by the construction of $R$. To check the initial value (1.19) that $(u,\;\rho,\;R)|_{t=0}=(u_0,\;\rho_0,\;R_0)$, let $\varphi\in C_c^\infty[0,+\infty)$ with $\text{supp }\varphi\subset[0,N]$. Then for $k>2N$, 
$$\begin{aligned}
\left(u(0)-u_0,\;\varphi\right)_{L^2}
=&\left(u(0)-u_{k,0},\;\varphi\right)_{L^2}\\
=&\frac{1}{T}\int_0^T\left(u(t)-u_k(t),\;\varphi\right)_{L^2}dt+\frac{1}{T}\int_0^T(t-T)\left(\partial_tu-\partial_tu_k,\;\varphi\right)_{L^2}dt\\
=&\frac{1}{T}\int_0^T\left(u(t)-\tilde{u}_k(t),\;\varphi\right)_{L^2}dt+\frac{1}{T}\int_0^T(t-T)\left(\partial_tu-\partial_t\tilde{u}_k,\;\varphi\right)_{L^2}dt\\
\rightarrow & 0,\text{ as }k\rightarrow+\infty.
\end{aligned}$$
Hence $u(t=0)=u_0$, and similarly $\rho(t=0)=\rho_0$. For any $N>0$, in view of Rellich's selection theorem, there exists a subsequence of $(\tilde{u}_k,\;\tilde{\rho}^{-1}_k)$, still denoted by $(\tilde{u}_k,\;\tilde{\rho}^{-1}_k)$, such that
\begin{equation}
(\tilde{u}_k,\;\tilde{\rho}^{-1}_k)\rightarrow (u,\rho^{-1}),\;\text{strongly in }L^2\left((0,N)\times(0,T)\right).
\end{equation}
Then by (3.3), $(u_k,\;\rho_k)$ also converges strongly to $(u,\;\rho)$ in  $L^2\left((0,N)\times(0,T)\right)$, and thus
\begin{equation}
r_k\rightarrow r,\text{ strongly in }C\left([0,T],L^2(0,N)\right).
\end{equation}
Therefore, in view of (3.3)(3.13), for any $t\in[0,T]$, 
\begin{equation}\begin{aligned}
R_k(t)-R(t)
=&\int_0^t(u_k-u)|_{x=0}d\tau\\
=&\frac{1}{N}\int_0^t\int_0^N(u_k-u)dxd\tau+\frac{1}{N}\int_0^t\int_0^N(x-N)(\partial_xu_k-\partial_xu)dxd\tau\\
\rightarrow& 0,\;\text{as }k\rightarrow+\infty
\end{aligned}\end{equation}
In particular, $R(0)=\lim_{k\rightarrow+\infty}R_k(0)=R_0$, which verifies (1.19).\\ 
\indent To check (1.14), let $\psi\in C_c^\infty(Q_T)$ with $\text{supp}\psi\subset (0,N)\times(0,T)$. 
First by the weak convergence (3.13) and (3.3), one has
$$\int_{Q_T}\partial_t\rho^{-1}\psi dxdt=\lim_{k\rightarrow+\infty}
\int_{Q_T}\partial_t\tilde{\rho}^{-1}_k\psi dxdt=\lim_{k\rightarrow+\infty}
\int_{Q_T}\partial_t\rho^{-1}_k\psi dxdt.$$
Then using equation (2.1) and integrating by parts, it holds
$$\lim_{k\rightarrow+\infty}
\int_{Q_T}\partial_t\rho^{-1}_k\psi dxdt=\lim_{k\rightarrow+\infty}\int_{Q_T}(r^2_ku_k)_x\psi dxdt=-\lim_{k\rightarrow+\infty}\int_{Q_T}r_k^2u_k\psi_xdxdt$$
Next, using the strong convergence (3.16)(3.17) and integrating by parts again, it follows 
$$
\int_{Q_T}\partial_t\rho^{-1}\psi dxdt=-\lim_{k\rightarrow+\infty}\int_{Q_T}r_k^2u_k\psi_xt=-\lim_{k\rightarrow+\infty}\int_{Q_T}r^2u\psi_xdxdt=\int_{Q_T}(r^2u)_x\psi dxdt
$$
Hence $\partial_t\rho^{-1}=(r^2u)_x$, which is exactly (1.14). Moreover, it follows from 
$\frac{1}{3}\partial_t\partial_xr^3=\partial_x(r^2u)=\partial_t\rho^{-1}$ that
\begin{equation}
\partial_x r^3=3\rho^{-1}-3\rho^{-1}_0+\partial_xr_0^3=3\rho^{-1},\;\partial_xr=\rho^{-1}r^{-2},
\end{equation}
$$
r^3(x,t)=R^3(t)+3\int_0^x\rho^{-1}(y,t)dy,
$$
which together with the construction of $r$ verify (1.18).
Since $c(T)\leq\rho\leq C(T)$, it also follows that
\begin{equation}
c(T)\leq(1+3x)^{-1}r^3\leq C(T).
\end{equation}
\indent To check (1.15), write the inner product of the viscous term with an arbitrary function $\psi$ as
$$\begin{aligned}
\int_{Q_T}(\rho_k(r_k^2u_k)_x)_xr_k^2\psi dxdt
=&-\int_{Q_T}\rho_k(r_k^2u_k)_x(r_k^2\psi)dxdt\\
=&\int_{Q_T}(\rho-\rho_k)(r_k^2u_k)_x(r_k^2\psi)_xdxdt-\int_{Q_T}\rho(r_k^2u_k)_x(r^2\psi)_xdxdt\\
&-\int_{Q_T}\rho(r_k^2u_k)_x[(r_k^2-r^2)\psi]_xdxdt.
\end{aligned}$$
By the strong convergence of $\tilde{\rho}^{-1}_k$ (3.16), the bounds of $\rho$ (3.15), $\rho_k$ (3.5), $r_k$ (3.7), $\tilde{u}_k$ (3.12) in view of (3.3) and the compact support of $\psi$, the first term on the right-hand side tends to 0 as $k\rightarrow+\infty$. Similarly, the third term vanishes as $k\rightarrow+\infty$ by (3.3)(3.7)(3.12)(3.15)(3.16)(3.17) and the bounds of $r$ (3.20), while the second term tends to $-\int_{Q_T}\rho(r^2u)_x(r^2\psi)_xdxdt=\int_{Q_T}(\rho(r^2u)_x)_xr^2\psi dxdt$ by (3.13)(3.15)(3.19)(3.20) and the equation (1.14). Hence we find that 
\begin{equation}
\lim_{k\rightarrow+\infty}\int_{Q_T}(\rho_k(r_k^2u_k)_x)_xr_k^2\psi dxdt=\int_{Q_T}(\rho(r^2u)_x)_xr^2\psi dxdt.
\end{equation}
For the pressure term, write 
$$\begin{aligned}
\int_{Q_T}(\rho_k^\gamma)_xr_k^2\psi dxdt=&-\int_{Q_T}(\rho_k^\gamma-1)(r_k^2\psi)_xdxdt\\
=&\int_{Q_T}(\rho^\gamma-\rho_k^\gamma)(r_k^2\psi)_xdxdt-\int_{Q_T}(\rho^\gamma-1)(r^2\psi)_xdxdt\\
&-\int_{Q_T}(\rho^\gamma-1)((r_k^2-r^2)\psi)_xdxdt.
\end{aligned}$$
(3.3)(3.5)(3.7)(3.15)(3.16) imply that the first term tends to 0, and (3.3)(3.7)(3.14)(3.15)(3.16) (3.17)(3.20) imply that the third term tends to 0. Hence 
\begin{equation}
\lim_{k\rightarrow+\infty}\int_{Q_T}(\rho_k^\gamma)_xr_k^2\psi dxdt=\int_{Q_T}(\rho^\gamma)_xr^2\psi dxdt,
\end{equation}
(3.3)(3.13)(3.21)(3.22) then imply that
$$\begin{aligned}
0=&\lim_{k\rightarrow+\infty}\int_{Q_T}\left[\partial_tu_k+\frac{Ca}{2}(\rho_k^\gamma)_xr_k^2-\mu r_k^2(\rho_k(r_k^2u_k)_x)_x\right]\psi dxdt
\\&=\int_{Q_T}\left[\partial_tu+\frac{Ca}{2}(\rho^\gamma)_xr^2-\mu r^2(\rho(r^2u)_x)_x\right]\psi dxdt,
\end{aligned}$$
which verifies equation (1.15).  \\
\indent At last, to check (1.17), take $\psi\in C_c^\infty(Q_T)$ such that $\text{supp}\psi\subset[0,N)\times(0,T)$. Then by (1.14),
$$\begin{aligned}
&-\int_0^T\left.\left(\frac{Ca}{2}\rho^\gamma-\mu\rho r^2u_x\right)\right|_{x=0}\psi|_{x=0}dt\\
=&\int_{Q_T}\left(\frac{Ca}{2}\rho^\gamma-\mu\rho r^2u_x\right)\psi_xdxdt+\int_{Q_T}\left(\frac{Ca}{2}\rho^\gamma-\mu\rho r^2u_x\right)_x\psi dxdt\\
=&\int_{Q_T}\left(\frac{Ca}{2}\rho^\gamma-\mu\rho r^2u_x\right)\psi_xdxdt+\int_{Q_T}\left(\frac{Ca}{2}(\rho^\gamma)_x+\mu(\log\rho)_{xt}+2\mu\frac{u_x}{r}+2\mu\rho^{-1}\frac{u}{r^4}\right)\psi dxdt,
\end{aligned}$$
and the same equation holds with $(u,\;\rho,\;r)$ replaced by $(u_k,\;\rho_k,\;r_k)$. By (3.3)(3.5)(3.7)(3.13)(3.15) (3.16)(3.17)(3.20) and the convergence of $R_k$ (3.18), (1.17) is verified by
$$\begin{aligned}
&\int_0^T\left.\left(\frac{Ca}{2}\rho^\gamma-\mu\rho r^2u_x\right)\right|_{x=0}\psi|_{x=0}dt\\
=&\lim_{k\rightarrow+\infty}\int_0^T\left.\left(\frac{Ca}{2}\rho_k^\gamma-\mu\rho_k r_k^2(u_k)_x\right)\right|_{x=0}\psi|_{x=0}dt\\
=&\lim_{k\rightarrow+\infty}\int_0^T\left(\left(\frac{Ca}{2}+\frac{2}{We}\right)R_k^{-3\gamma_0}-\frac{2}{We}R_k^{-1}\right)\psi|_{x=0}dt\\
=&\int_0^T\left(\left(\frac{Ca}{2}+\frac{2}{We}\right)R^{-3\gamma_0}-\frac{2}{We}R^{-1}\right)\psi|_{x=0}dt.
\end{aligned}$$
Hence $(u,\;\rho,\;R)$ is a generalized solution to (1.14-1.19) on $[0,T]$. Since $T>0$ is chosen arbitrarily, we conclude that $(u,\;\rho,\;R)$ is in fact a global generalized solution.
\subsection{Uniqueness}
Let $(u_1,\;\rho_1,\;R_1)$, $(u_2,\;\rho_2,\;R_2)$ be two global generalized solutions to (1.14-1.19), namely, for $i=1,2$,
\begin{numcases}
{} \partial_t\rho_i+\rho_i^2\partial_x(r_i^2u_i)=0, &$x>0,\;t>0,$\\
\partial_t u_i+\frac{Ca}{2}r_i^2\partial_x(\rho_i^\gamma) =\mu r_i^2\partial_x\left(\rho_i\partial_x(r_i^2u_i)\right), &$x>0,\;t>0,$\\
\frac{dR_i}{dt}=u_i|_{x=0}, &$t>0,$\\
(\frac{Ca}{2}\rho_i^\gamma-\mu\rho_i r_i^2\partial_x u_i)|_{x=0}=\left(\frac{Ca}{2}+\frac{2}{We}\right)R_i^{-3\gamma_0}-\frac{2}{We}R_i^{-1}, &$t>0,$\\
r_i=\left(R_i(t)^3+3\int_0^x\rho_i^{-1}(y,t)dy\right)^{\frac{1}{3}}=r_i(x,0)+\int_0^t u_i(y,\tau)d\tau, &$x>0,\;t>0,$\\
(u_i,\;\rho_i,\;R_i)|_{t=0}=(u_0,\;\rho_0,\;R_0), &$x>0$,
\end{numcases} 
and the control (3.14)(3.15)(3.20) hold.
To prove the uniqueness, it suffices to show $(u_1,\;\rho_1,\;R_1)=(u_2,\;\rho_2,\;R_2)$ on $[0,T]$ for arbitrary $T>0$.
To begin with, subtracting (3.24) with $i=1,2$, and multiplying the resulted equation by $(u_1-u_2)$ yield the equation 
\begin{equation}\begin{aligned}
0=&\frac{1}{2}\frac{d}{dt}(u_1-u_2)^2+\frac{Ca}{2}\left[(\rho_1^\gamma)_xr_1^2-(\rho_2^\gamma)_xr_2^2\right](u_1-u_2)\\
&-\mu\left[(\rho_1(r_1^2u_1)_x)_xr_1^2-(\rho_2(r_2^2u_2)_x)_xr_2^2\right](u_1-u_2)\\
=&\frac{1}{2}\frac{d}{dt}(u_1-u_2)^2+\frac{Ca}{2}\left[(\rho_1^\gamma r_1^2-\rho_2^\gamma r_2^2)(u_1-u_2)\right]_x-\frac{Ca}{2}(\rho_1^\gamma r_1^2-\rho_2^\gamma r_2^2)(u_1-u_2)_x\\
&-\frac{Ca}{2}(\rho_1^2(r_1^2)_x-\rho_2^2(r_2^2)_x)(u_1-u_2)-\mu\left[(\rho_1(r_1^4(u_1)_x-\rho_2(r_2^4(u_2)_x)(u_1-u_2)\right]_x\\
&+\mu\rho_2 r_2^4(u_1-u_2)_x^2+2\mu\rho_2^{-1}r_2^{-2}(u_1-u_2)^2\\
&+\mu(\rho_1 r_1^4-\rho_2 r_2^4)(u_1)_x(u_1-u_2)_x+2\mu(\rho_1^{-1}r_1^{-2}-\rho_2^{-1}r_2^{-2})u_1(u_1-u_2).
\end{aligned}\end{equation}
Using (3.26), integrating (3.29) on $(0,+\infty)$ yields that
\begin{equation}\begin{aligned}
&\frac{1}{2}\frac{d}{dt}\int_0^\infty(u_1-u_2)^2dx-\frac{Ca}{2}\int_0^\infty(\rho_1^\gamma(r_1^2)_x-\rho_2^\gamma(r_2^2)_x)(u_1-u_2)dx\\
&-\frac{Ca}{2}\int_0^\infty(\rho_1^\gamma r_1^2-\rho_2^\gamma r_2^2)(u_1-u_2)_xdx+\mu\int_0^\infty\rho_2r_2^4(u_1-u_2)_x^2dx+2\mu\int_0^\infty\rho_2^{-1}r_2^{-2}(u_1-u_2)^2dx\\
&+\mu\int_0^\infty(\rho_1 r_1^4-\rho_2r_2^4)(u_1)_x(u_1-u_2)_xdx+2\mu\int_0^\infty(\rho_1^{-1}r_1^{-2}-\rho_2^{-1}r_2^{-2})u_1(u_1-u_2)dx\\
&+\left[\frac{2}{We}(R_1-R_2)-\left(\frac{Ca}{2}+\frac{2}{We}\right)(R_1^{-3\gamma_0+2}-R_2^{-3\gamma_0+2})\right]\left(\frac{dR_1}{dt}-\frac{dR_2}{dt}\right)=0.
\end{aligned}\end{equation}
Since $$(R_2-R_1)^{-1}\left(R_2^{-3\gamma_0+2}-R_1^{-3\gamma_0+2}\right)=\int_0^1(2-3\gamma_0)\left(R_1+\lambda(R_2-R_1)\right)^{1-3\gamma_0}d\lambda,$$ 
the difference of $R_i$ can be bounded by
\begin{equation}\begin{aligned}
&\left[\frac{2}{We}(R_1-R_2)-\left(\frac{Ca}{2}+\frac{2}{We}\right)(R_1^{-3\gamma_0+2}-R_2^{-3\gamma_0+2})\right]\left(\frac{dR_1}{dt}-\frac{dR_2}{dt}\right)\\
=&\frac{d}{dt}\left[\left(\frac{1}{2}\left(\frac{Ca}{2}+\frac{2}{We}\right)\frac{R_2^{-3\gamma_0+2}-R_1^{-3\gamma_0+2}}{R_1-R_2}+\frac{1}{We}\right)(R_1-R_2)^2\right]\\
&+\frac{1}{2}\left(\frac{Ca}{2}+\frac{2}{We}\right)(R_1-R_2)^2\frac{d}{dt}\int_0^1(2-3\gamma_0)\left(R_1+\lambda(R_2-R_1)\right)^{1-3\gamma_0}d\lambda\\
=&\frac{d}{dt}\left[\left(\frac{1}{2}\left(\frac{Ca}{2}+\frac{2}{We}\right)\frac{R_2^{-3\gamma_0+2}-R_1^{-3\gamma_0+2}}{R_1-R_2}+\frac{1}{We}\right)(R_1-R_2)^2\right]\\
&+\frac{1}{2}\left(\frac{Ca}{2}+\frac{2}{We}\right)(R_1-R_2)^2(2-3\gamma_0)(1-3\gamma_0)\int_0^1\left(R_1+\lambda(R_2-R_1)\right)^{-3\gamma_0}(u_1+\lambda(u_2-u_1))d\lambda\\
\geq&\frac{d}{dt}\left[\left(\frac{1}{2}\left(\frac{Ca}{2}+\frac{2}{We}\right)\frac{R_2^{-3\gamma_0+2}-R_1^{-3\gamma_0+2}}{R_1-R_2}+\frac{1}{We}\right)(R_1-R_2)^2\right]-C(T)(R_1-R_2)^2,
\end{aligned}\end{equation}
where in the last step we use that $\|u_i\|_{L^\infty}\leq C(T)\left(\|u_i\|_{L^2}+\|r_i^2(u_i)_x\|_{L^2}\right)\leq C(T)$ in view of (3.14)(3.15)(3.20) and that $c\leq R_i\leq C$ for $i=1,2$. 
Then using Cauchy-Schwarz inequality and (3.14)(3.15)(3.20), the two terms with coefficient $\frac{Ca}{2}$ in (3.30) have the control that
\begin{equation}\begin{aligned}
&-\frac{Ca}{2}\int_0^\infty(\rho_1^\gamma(r_1^2)_x-\rho_2^\gamma(r_2^2)_x)(u_1-u_2)dx-\frac{Ca}{2}\int_0^\infty(\rho_1^\gamma r_1^2-\rho_2^\gamma r_2^2)(u_1-u_2)_xdx\\
&+\mu\int_0^\infty(\rho_1 r_1^4-\rho_2r_2^4)(u_1)_x(u_1-u_2)_xdx+2\mu\int_0^\infty(\rho_1^{-1}r_1^{-2}-\rho_2^{-1}r_2^{-2})u_1(u_1-u_2)dx\\
\geq&-\frac{\mu}{2}\int_0^\infty\rho_2r_2^4(u_1-u_2)_x^2dx-\mu\int_0^\infty\rho_2^{-1}r_2^{-2}(u_1-u_2)^2dx\\
&-C\int_0^\infty(\rho_1^\gamma r_1^2-\rho_2^\gamma r_2^2)^2\rho_2^{-1}r_2^{-4}dx-C\int_0^\infty(\rho_1^\gamma(r_1^2)_x-\rho_2^\gamma(r_2^2)_x)^2\rho_2r_2^2dx\\
&-C\int_0^\infty(\rho_1r_1^4-\rho_2r_2^4)^2(u_1)_x^2\rho_2^{-1}r_2^{-4}dx-C\int_0^\infty(\rho_1^{-1}r_1^{-2}-\rho_2^{-1}r_2^{-2})^2\rho_2r_2^2u_1^2dx.\\
\geq&-\frac{\mu}{2}\int_0^\infty\rho_2r_2^4(u_1-u_2)_x^2dx-\mu\int_0^\infty\rho_2^{-1}r_2^{-2}(u_1-u_2)^2dx\\
&-C(T)\int_0^\infty(\rho_1^{-1}-\rho_2^{-1})^2dx-C(T)\int_0^\infty(1-\frac{r_1}{r_2})^2dx-C(T)\int_0^\infty(1-\frac{r_2}{r_1})^2dx,
\end{aligned}\end{equation}
where we also used that \begin{equation}
\|u_1\|_{L^\infty}^2\leq C(T)\left(\|u_1\|_{L^2}^2+\|r_1^2(u_1)_x\|_{L^2}^2\right)\leq C(T),
\end{equation}\begin{equation}
\|\rho_1(r_1^2u_1)_x\|_{L^\infty}\leq C(T)(1+\|(u_1)_t\|_{L^2})\leq C(T).
\end{equation} 
(3.34) is verified by dividing (3.24) with $r_1^2$ and integrating on $[x,+\infty)$, namely
$$\mu\rho_1(r_1^2u_1)_x(x,t)=-\int_x^\infty\frac{(u_1)_t}{r_1^2}dy+\frac{Ca}{2}(\rho_1^\gamma(x,t)-1).$$
The third line in (3.30) can be absorbed by  using Cauchy-Schwarz inequality and (3.33)(3.34),  therefore  using (3.31)(3.32) in (3.30) yields 
\begin{equation}\begin{aligned}
&\frac{1}{2}\frac{d}{dt}\int_0^\infty(u_1-u_2)^2dx+\frac{d}{dt}\left[\left(\frac{1}{2}\left(\frac{Ca}{2}+\frac{2}{We}\right)\frac{R_2^{-3\gamma_0+2}-R_1^{-3\gamma_0+2}}{R_1-R_2}+\frac{1}{We}\right)(R_1-R_2)^2\right]\\
&+\frac{\mu}{2}\int_0^\infty\rho_2r_2^4(u_1-u_2)_x^2dx+\mu\int_0^\infty\rho_2^{-1}r_2^{-2}(u_1-u_2)^2dx\\
\leq&C(T)\left[\int_0^\infty(\rho_1^{-1}-\rho_2^{-1})^2dx+\int_0^\infty(1-\frac{r_1}{r_2})^2dx+\int_0^\infty(1-\frac{r_2}{r_1})^2dx+(R_1-R_2)^2\right].
\end{aligned}\end{equation} 
To close (3.35), it remains to control $\|\rho_1^{-1}-\rho_2^{-1}\|_{L^2}$ and $\|1-r_i^{-2}r_j^2\|_{L^2}$. In fact, for the difference of $\rho_i^{-1}$, we have the estimate that
\begin{equation}\begin{aligned}
&\frac{d}{dt}\int_0^\infty(\rho_1^{-1}-\rho_2^{-2})^2dx=2\int_0^\infty(\rho_1^{-1}-\rho_2^{-2})\left((r_1^2u_1)_x-(r_2^2u_2)_x\right)dx\\
\leq&\int_0^\infty(\rho_1^{-1}-\rho_2^{-1})^2dx+2\int_0^\infty\left(r_1^2(u_1)_x-r_2^2(u_2)_x\right)^2dx+2\int_0^\infty\left((r_1^2)_xu_1-(r_2^2)_xu_2\right)^2dx\\
\leq&\int_0^\infty(\rho_1^{-1}-\rho_2^{-1})^2dx+4\int_0^\infty(r_1^2(u_1)_x)^2\left(1-\frac{r_2}{r_1}\right)^2dx+4\int_0^\infty r_2^4(u_1-u_2)_x^2dx\\
&+16\int_0^\infty u_1^2(\rho_1^{-1}r_1^{-1}-\rho_2^{-1}r_2^{-1})^2dx+16\int_0^\infty\rho_2^{-2}r_2^{-2}(u_1-u_2)^2dx\\
\leq&C(T)\left[\int_0^\infty(\rho_1^{-1}-\rho_2^{-1})^2dx+\int_0^\infty\left(1-\frac{r_2}{r_1}\right)^2dx+\int_0^\infty(u_1-u_2)^2dx+\int_0^\infty\rho_2r_2^4(u_1-u_2)_x^2dx\right].
\end{aligned}\end{equation}
Using $\partial_t r_i=u_i$ for $i=1,2$, $\|1-r_i^{-2}r_j^2\|_{L^2}$ can be bounded easily:
\begin{equation}\begin{aligned}
&\frac{d}{dt}\int_0^\infty\left(1-\frac{r_2}{r_1}\right)^2dx\\
=&2\int_0^\infty\left(\frac{r_2}{r_1}-1\right)\left(\frac{r_1u_2-r_2u_1}{r_1^2}\right)dx\\
=&-2\int_0^\infty\left(1-\frac{r_2}{r_1}\right)^2\frac{u_2}{r_1}dx+2\int_0^\infty\frac{r_2}{r_1^2}\left(\frac{r_2}{r_1}-1\right)(u_2-u_1)dx\\
\leq&C(T)\int_0^\infty\left(1-\frac{r_2}{r_1}\right)^2dx+C(T)\int_0^\infty(u_1-u_2)^2dx,
\end{aligned}\end{equation}
and in the same way,
\begin{equation}\begin{aligned}
\frac{d}{dt}\int_0^\infty\left(1-\frac{r_1}{r_2}\right)^2dx\leq C(T)\int_0^\infty\left(1-\frac{r_1}{r_2}\right)^2dx+C(T)\int_0^\infty(u_1-u_2)^2dx.
\end{aligned}\end{equation}
To cancel the bad term $\int_0^\infty\rho_2 r_2^4(u_1-u_2)_x^2dx$ in (3.36), choose a small enough $\epsilon(T)>0$, we conclude by (3.35)(3.36)(3.37)(3.38) that
$$\begin{aligned}
&\frac{d}{dt}\int_0^\infty(u_1-u_2)^2dx+\frac{d}{dt}\left[\left(\frac{1}{2}\left(\frac{Ca}{2}+\frac{2}{We}\right)\frac{R_2^{-3\gamma_0+2}-R_1^{-3\gamma_0+2}}{R_1-R_2}+\frac{1}{We}\right)(R_1-R_2)^2\right]\\
&+\epsilon(T)\frac{d}{dt}\int_0^\infty(\rho_1^{-1}-\rho_2^{-2})^2dx+\frac{d}{dt}\int_0^\infty\left(1-\frac{r_2}{r_1}\right)^2dx+\frac{d}{dt}\int_0^\infty\left(1-\frac{r_1}{r_2}\right)^2dx\\
\leq& C(T)\left[\int_0^\infty(u_1-u_2)^2dx+\int_0^\infty(\rho_1^{-1}-\rho_2^{-1})^2dx+(R_1-R_2)^2\right]\\
&+C(T)\int_0^\infty\left((1-\frac{r_1}{r_2})^2+(1-\frac{r_2}{r_1})^2\right)dx.
\end{aligned}$$
Gronwall's inequality then shows that $(u_1,\;\rho_1,\;R_1)=(u_2,\;\rho_2,\;R_2)$.
\section{Uniform estimates and asymptotic stability}
In this section, uniform in time estimates are given for solutions with initial data closed to the equilibrium. Then Theorem 1.3 is proved by making full use of the dissipation terms in energy estimates. Let $(u,\;\rho,\;R)$ be the global generalized solution to (1.14)-(1.19) with $(u_0,\;\rho_0,\;R_0)$ satisfying the assumption (1.21). Let $(\tilde{u}_k,\;\tilde{\rho}_k,\;R_k)$ be the same as in Section 3. We first establish the same basic energy identity for $(u,\;\rho,\;R)$ as $(u_k,\;\rho_k,\;R_k)$ in Section 2.
\begin{lem}[Basic energy]
Let $P(R)$ be the same as in Lemma 2.2, and let $$E_0(t):=\frac{1}{2}\int_0^\infty u^2dx+\frac{Ca}{2}\frac{1}{\gamma-1}\int_0^\infty H(\rho)dx+P(R).$$ Then for any $t\in [0,T]$, 
\begin{equation}
E_0(t)+\mu\int_0^t\int_0^\infty\rho(r^2u_x)^2dxd\tau
+2\mu\int_0^t\int_0^\infty\rho^{-1}\frac{u^2}{r^2}dxd\tau=E_0(0).
\end{equation}
\end{lem}
\begin{proof}
Using (1.14) and (1.15), a direct calculation gives that
$$\begin{aligned}
&\frac{1}{2}\int_0^\infty u^2dx+\frac{Ca}{2}\frac{1}{\gamma-1}\int_0^\infty H(\rho)dx+P(R)+\mu\int_0^t\int_0^\infty\rho(r^2u_x)^2dxd\tau
+2\mu\int_0^t\int_0^\infty\rho^{-1}\frac{u^2}{r^2}dxd\tau\\
=&\int_0^t\int_0^\infty\left\{uu_t+\frac{Ca}{2}(\rho^\gamma-1)\rho^{-2}\rho_t+\rho(r^2u_x)^2+\rho^{-1}\frac{u^2}{r^2}\right\}dxd\tau+\int_0^t\frac{d}{dt}P(R)d\tau+E_0(0)\\
=&-\frac{Ca}{2}\int_0^t\int_0^\infty\left[(\rho^\gamma-1)_xr^2u+(\rho^\gamma-1)(r^2u)_x\right]dxd\tau+\mu\int_0^t\int_0^\infty\left[(\rho(r^2u)_x)_xr^2u+\rho(r^2u)_x^2\right]dxd\tau\\
&-2\mu\int_0^t\int_0^\infty\left[2ruu_x+\rho^{-1}\frac{u^2}{r^2}\right]dxd\tau+\int_0^t\frac{d}{dt}P(R)d\tau+E_0(0)
\end{aligned}$$
Note that by (3.13), 
$$-\frac{Ca}{2}\int_0^t\int_0^\infty\left[(\rho^\gamma-1)_xr^2u+(\rho^\gamma-1)(r^2u)_x\right]dxd\tau=-\frac{Ca}{2}\lim_{k\rightarrow+\infty}\int_0^t\int_0^\infty\left[(\rho^\gamma-1)r^2\tilde{u}_k\right]_xdxd\tau,$$
$$\mu\int_0^t\int_0^\infty\left[(\rho(r^2u)_x)_xr^2u+\rho(r^2u)_x^2\right]dxd\tau=\mu\lim_{k\rightarrow+\infty}\int_0^t\int_0^\infty\left[(\rho(r^2u)_x)r^2\tilde{u}_k\right]_xdxd\tau,$$
$$-2\mu\int_0^t\int_0^\infty\left[2ruu_x+\rho^{-1}\frac{u^2}{r^2}\right]dxd\tau=-2\mu\lim_{k\rightarrow+\infty}\int_0^t\int_0^\infty\left[ru\tilde{u}_k\right]_xdxd\tau,$$
and that in view of the boundary conditions (1.16)(1.17) the sum of the right-hand sides of the above three equations is 
$$\lim_{k\rightarrow+\infty}\int_0^t\left[\left(\frac{Ca}{2}+\frac{2}{We}\right)R^{-3\gamma_0+2}-\frac{Ca}{2}R^2-\frac{2}{We}R\right]\frac{dR_k}{dt}d\tau,$$
which is exactly $-\int_0^t\frac{d}{dt}P(R)d\tau$ since $\frac{dR_k}{dt}=-\int_0^\infty(\tilde{u}_k)_xdx$ and $(1+3x)^{\frac{2}{3}}(\tilde{u}_k)_x\rightharpoonup (1+3x)^{\frac{2}{3}}u_x$ in the weak-$\ast$ sense of $L^\infty([0,T],L^2)$. Hence all the terms on the right-hand side of the energy identity are cancelled except $E_0(0)$.
\end{proof}
From Lemma 4.1, we see that $P(R)\leq\ E_0$, and the convexity of $P(R)$ yields for some positive constant $C$ that
 \begin{equation}
(R-1)^2\leq CE_0.
\end{equation}  
Similar to Lemma 2.3, a corollary is that 
\begin{equation}
\int_0^\infty\|u^2r\|_{L^\infty}dt\leq\mu^{-1}E_0(0).
\end{equation} As in Lemma 2.5, define for $t>0$ that
$$E_1(t):=\frac{1}{2}\int_0^\infty\left(u+\mu r^2(\log\rho)_x\right)^2dx+\frac{Ca}{2}\frac{1}{\gamma-1}\int_0^\infty H(\rho)dx+P(R).$$
Then a same calculation gives that 
\begin{lem}[Bresch-Desjardins entropy equation]
\begin{equation}
\frac{d}{dt}E_1(t)+\frac{Ca}{2}\frac{4\mu}{\gamma}\int_0^\infty\left(r^2(\rho^{\frac{\gamma}{2}})_x\right)^2dx=2\mu\int_0^\infty(\log\rho)_xru(u+\mu(\log\rho)_xr^2)dx+\mu(\rho r^2u_x)|_{x=0}R^2\frac{dR}{dt}.
\end{equation}
\end{lem}
\indent Using Lemma 4.1 and Lemma 4.2, we derive a uniform in time estimate of $E_1$ provided that the initial data is close to the equilibrium in the sense that $E_0(0)+E_1(0)\leq\delta$ for $\delta\lesssim 1$ sufficiently small. 
\begin{lem}
There exists $\delta>0$ such that if the initial data $(u_0,\;\rho_0,\;R_0)$ is close to the equilibrium in the sense that $E_0(0)+E_1(0)\leq\delta$, then there exists $C>0$ such that \\
(i) $E_0(t)+E_1(t)\leq C(E_0(0)+E_1(0))$ for any $t>0$,\\
(ii) $\int_0^\infty\int_0^\infty\left(r^2(\rho^\frac{\gamma}{2})_x\right)^2dxdt\leq C(E_0(0)+E_1(0)),$\\
(iii) $\rho\leq 1+C(E_0(0)+E_1(0))^\frac{1}{2},\;\rho^{-1}\leq 1+C(E_0(0)+E_1(0))^\frac{1}{2}$ for any $(x,t)\in\mathbb{R}^+\times\mathbb{R}^+$.
\end{lem}
\begin{proof}
We begin with the estimates of the right-hand side of (4.4). 
Using (1.18) and $\int_0^\infty r^{-4}dx=\int_{R(t)}^\infty\rho r^{-2}dr\leq R(t)^{-1}\sup_{x\in\mathbb{R}}\rho$, it holds that
\begin{equation}\begin{aligned}
\left|\int_0^\infty\mu ru^2(\log\rho)_xdx\right|\leq &\left(\frac{Ca\gamma}{4\mu}\right)^{-1}\int_0^\infty\rho^{-\gamma}\frac{u^4}{r^2}dx+\frac{Ca}{2}\frac{\mu}{2\gamma}\int_0^\infty(\rho^\frac{\gamma}{2})_x^2r^4dx\\
\leq&\left(\frac{Ca\gamma}{4\mu}\right)^{-1}\|u^2r\|_{L^\infty}\|r^{-3}\rho^{-\gamma}\|_{L^\infty}\int_0^\infty u^2dx+\frac{Ca}{2}\frac{\mu}{2\gamma}\int_0^\infty(\rho^\frac{\gamma}{2})_x^2r^4dx,
\end{aligned}\end{equation}
and
\begin{equation}\begin{aligned}
&\left|\int_0^\infty\frac{u}{r}(\mu(\log\rho)_xr^2)^2dx\right|\\
\leq&\left(\frac{Ca\gamma}{4\mu}\right)^{-1}\int_0^\infty\frac{u^2}{r^2}\rho^{-\gamma}(\mu(\log\rho)_xr^2)^2dx+\frac{Ca}{2}\frac{\mu}{2\gamma}\int_0^\infty(\rho^\frac{\gamma}{2})_x^2r^4dx\\
\leq&\left(\frac{Ca\gamma}{4\mu}\right)^{-1}\|u^2r\|_{L^\infty}\|r^{-3}\rho^{-\gamma}\|_{L^\infty}\int_0^\infty(\mu(\log\rho)_xr^2)^2dx+\frac{Ca}{2}\frac{\mu}{2\gamma}\int_0^\infty(\rho^\frac{\gamma}{2})_x^2r^4dx.
\end{aligned}\end{equation}
To control the lower bound of $\rho$, note first that
$$\left|\partial_x(1-\rho^{-\frac{1}{4}})^2\right|=\left|\frac{1}{2}(1-\rho^{-\frac{1}{4}})\rho^{-\frac{1}{4}}(\log\rho)_x\right|\leq\frac{1}{4}(\rho^{-\frac{1}{4}}-\rho^{-\frac{1}{2}})^2r^{-4}+\frac{1}{4}(\log\rho)_x^2r^4.$$
In order to control the first term on the right-hand side, we use the inequality 
$$(\rho^{\lambda_1}-\rho^{\lambda_2})^2\leq C_{\lambda_1,\lambda_2} H(\rho)\leq C_{\lambda_1,\lambda_2}E_0$$ 
for $\lambda_1,\;\lambda_2$ with $-\frac{1}{2}\leq\lambda_2\leq\lambda_1\leq\frac{\gamma-1}{2}$.
Hence by integrating in $x$ we obtain  the lower bound of $\rho$ that $\|1-\rho^{-\frac{1}{4}}\|^2_{L^\infty}\leq C(E_0+E_1)$. Similarly, for the upper bound, the inequality
$$\left|\partial_x(\rho^{\frac{\gamma-1}{4}}-1)^2\right|\leq\left|\frac{\gamma-1}{2}(\rho^{\frac{\gamma-1}{4}}-1)\rho^{\frac{\gamma-1}{4}}(\log\rho)_x\right|\leq\frac{\gamma-1}{2}(\rho^{\frac{\gamma-1}{2}}-\rho^{\frac{\gamma-1}{4}})^2r^{-4}+\frac{\gamma-1}{2}(\log\rho)_x^2r^4$$
yields that $\|\rho^{\frac{\gamma-1}{4}}-1\|^2_{L^\infty}\leq C(E_0+E_1)$. Hence 
\begin{equation}
\|\rho\|_{L^\infty}\leq\left(1+C(E_0+E_1)^\frac{1}{2}\right)^\frac{4}{\gamma-1},\;\|\rho^{-\gamma}\|_{L^\infty}\leq\left(1+C(E_0+E_1)^\frac{1}{2}\right)^{4\gamma}.
\end{equation}
Note that $(u,\;\rho,\;R)$ also satisfies (2.10), and thus $(\rho|_{x=0}R^2)^{-\gamma}$ can be represented by (2.11).  Let $S(t)$ be  as in Lemma 2.4.  Since $|R-1|^2\leq CP(R)\leq C\delta$, it holds for sufficiently small $\delta$ that 
\begin{equation}\left(\frac{Ca}{2}+\frac{2}{We}\right)R^{-3\gamma_0}-\frac{2}{We}R^{-1}\geq\frac{Ca}{2}-C|R-1|\geq\frac{Ca}{2}-C\delta^{\frac{1}{2}}>0,
\end{equation}
which guarantees that $\int_0^\infty S(t)dt<+\infty$.
Introduce the notation $$\mathfrak{R}(t)=\frac{Ca}{2}R^{-2\gamma}\left[\left(\frac{Ca}{2}+\frac{2}{We}\right)R^{-3\gamma_0}-\frac{2}{We}R^{-1}\right]^{-1}(t).$$ 
It then follows from (4.2) that $(\mathfrak{R}-1)^2\leq CE_0$.  Integrating by parts in (2.11) yields that
\begin{equation}\begin{aligned}
(\rho|_{x=0}R^2)^{-\gamma}=&\mathfrak{R}(t)+\left((\rho_0|_{x=0}R_0^2)^{-\gamma}-\mathfrak{R}(0)\right)S(t)-\int_0^tS(t-\tau)\frac{d\mathfrak{R}}{dt}(\tau)d\tau\\
=&\mathfrak{R}(t)+\left((\rho_0|_{x=0}R_0^2)^{-\gamma}-\mathfrak{R}(t)\right)S(t)\\&+\frac{\gamma}{\mu}\int_0^t(\mathfrak{R}(\tau)-\mathfrak{R}(t))\left[\left(\frac{Ca}{2}+\frac{2}{We}\right)R^{-3\gamma_0}-\frac{2}{We}R^{-1}\right]S(t-\tau)d\tau.
\end{aligned}\end{equation} 
Hence $1-C\delta^{\frac{1}{2}}\leq(\rho|_{x=0}R^2)^{-\gamma}\leq 1+C\delta^{\frac{1}{2}}$ for all $t>0$ in view of the integrability and boundedness of $S(t)$.
Since $$\mu(\rho r^2u_x)|_{x=0}=-\mu(\rho^{-1}\rho_t-2 ur^{-1})|_{x=0}=\frac{\mu}{\gamma}(\rho|_{x=0}R^2)^{\gamma}\partial_t(\rho|_{x=0}R^2)^{-\gamma},$$
and from (2.10)(4.9) that 
$$\begin{aligned}
&\frac{\mu}{\gamma}(\rho|_{x=0}R^2)^{\gamma}\partial_t(\rho|_{x=0}R^2)^{-\gamma}\\
=&-(\rho|_{x=0}R^2)^\gamma\left[\left(\frac{Ca}{2}+\frac{2}{We}\right)R^{-3\gamma_0}-\frac{2}{We}R^{-1}\right]\left((\rho|_{x=0}R^2)^{-\gamma}-\mathfrak{R}(t)\right)\\
=&-(\rho|_{x=0}R^2)^\gamma\left[\left(\frac{Ca}{2}+\frac{2}{We}\right)R^{-3\gamma_0}-\frac{2}{We}R^{-1}\right]\left((\rho_0|_{x=0}R_0^2)^{-\gamma}-\mathfrak{R}(0)\right)S(t)\\
&+(\rho|_{x=0}R^2)^\gamma\left[\left(\frac{Ca}{2}+\frac{2}{We}\right)R^{-3\gamma_0}-\frac{2}{We}R^{-1}\right]\int_0^tS(t-\tau)\frac{d\mathfrak{R}}{dt}(\tau)d\tau ,
\end{aligned}$$
it follows 
\begin{equation}
\left|\mu(\rho r^2u_x)|_{x=0}R^2\frac{dR}{dt}\right|\leq C \left((E_0(0)+E_1(0))^\frac{1}{2}S(t)\left|\frac{dR}{dt}\right|+\int_0^tS(t-\tau)\left|\frac{d\mathfrak{R}}{dt}\right|(\tau)d\tau\left|\frac{dR}{dt}\right|\right),
\end{equation}
where we used that $\left|(\rho_0|_{x=0}R_0^2)^{-\gamma}-\mathfrak{R}(0)\right|\leq C(E_0(0)+E_1(0))^\frac{1}{2}$ since $E_0(0)+E_1(0)\lesssim 1$. Moreover, (4.8) implies $S(t)\leq\exp\left\{-\frac{\gamma}{\mu}\left(\frac{Ca}{2}-C\delta^{\frac{1}{2}}\right)t\right\}$, and thus $\int_0^\infty S(t)\left|\frac{dR}{dt}\right|dt\leq C\left(\int_0^\infty\left|\frac{dR}{dt}\right|^2dt\right)^{\frac{1}{2}}$. Using Young's inequality, it holds $$\int_0^\infty\int_0^tS(t-\tau)\left|\frac{d\mathfrak{R}}{dt}\right|(\tau)d\tau\left|\frac{dR}{dt}\right|(t)dt\leq C\int_0^\infty\left|\frac{dR}{dt}\right|^2dt.$$ 
Integrating (4.10) in $t$, the boundary term $\mu(\rho r^2u_x)|_{x=0}R^2\frac{dR}{dt}$ has the control
\begin{equation}\begin{aligned}
&\int_0^\infty\left|\mu(\rho r^2u_x)|_{x=0}R^2\frac{dR}{dt}\right|dt\\
\leq &C\left[(E_0(0)+E_1(0))^\frac{1}{2}\left(\int_0^\infty\left|\frac{dR}{dt}\right|^2dt\right)^{\frac{1}{2}}+\int_0^\infty\left|\frac{dR}{dt}\right|^2dt\right]\\
\leq& C\left[(E_0(0)+E_1(0))^\frac{1}{2}\left(\int_0^\infty\|u^2r\|_{L^\infty}dt\right)^{\frac{1}{2}}+\int_0^\infty\|u^2r\|_{L^\infty}dt\right].
\end{aligned}\end{equation}
Now conclude from (4.4)(4.5)(4.6)(4.7) that
\begin{equation}\begin{aligned}
&\frac{d}{dt}E_1+\frac{Ca}{2}\frac{2\mu}{\gamma}\int_0^\infty\left(r^2(\rho^{\frac{\gamma}{2}})_x\right)^2dx\\
\leq&\left|\mu(\rho r^2u_x)|_{x=0}R^2\frac{dR}{dt}\right|+C\|u^2r\|_{L^\infty}\left(1+(E_0+E_1)^{2\gamma}\right)\left(\int_0^\infty u^2dx\right)\\
&+C\|u^2r\|_{L^\infty}\left(1+(E_0+E_1)^{2\gamma}\right)\left(\int_0^\infty(\mu(\log\rho)_xr^2)^2dx\right)\\
\leq& \left|\mu(\rho r^2u_x)|_{x=0}R^2\frac{dR}{dt}\right|+C\|u^2r\|_{L^\infty}\left(1+(E_0+E_1)^{2\gamma}\right)(E_0+E_1).
\end{aligned}\end{equation}
In view of (4.3), (4.11), $\frac{d}{dt}E_0\leq0$ and that $E_0(0)+E_1(0)\leq\delta\lesssim 1$, adding $\frac{d}{dt}E_0$ on the left-hand side of (4.12) and using Gronwall's inequality, it follows
$$
E_0(t)+E_1(t)\leq C(E_0(0)+E_1(0)),\;\int_0^\infty\int_0^\infty\left(r^2(\rho^\frac{\gamma}{2})_x\right)^2dxdt\leq C(E_0(0)+E_1(0)).
$$ 
At last, using (4.7) in view of $E_0(0)+E_1(0)\leq\delta\lesssim 1$, $\rho$ has uniform bounds from above and below in time as given in (iii).
\end{proof}
\begin{rem}
Through the proof of Lemma 4.3, one can see that the restriction $E_0(0)+E_1(0)\leq\delta\lesssim 1$ is due to the requirement that $\int_0^\infty S(t) dt\leq+\infty$ and the feasibility of applying Gronwall's inequality to (4.13). 
\end{rem}
In order to derive the viscous damping, it remains to establish the energy estimates for the derivatives of the solution. To this end, we have the following two energy identities.
\begin{lem}[Energy identities for 1-order derivatives]
Define for $t>0$, $$E_2(t)=\frac{1}{2}\int_0^\infty u_t^2dx+\frac{Ca}{2}\frac{2\gamma}{(\gamma-1)^2}\int_0^\infty(\rho^\frac{\gamma-1}{2})_t^2dx+\left[\frac{3\gamma_0}{2}\left(\frac{Ca}{2}+\frac{2}{We}\right)-\frac{1}{We}\right]\left(\frac{dR}{dt}\right)^2,$$
$$E_3(t)=\frac{1}{2}\int_0^\infty\left(u_t+\mu(\log\rho)_{xt}r^2\right)^2dx+\frac{Ca}{2}\frac{2\gamma}{(\gamma-1)^2}\int_0^\infty(\rho^\frac{\gamma-1}{2})_t^2dx.$$
Then 
\begin{equation}\begin{aligned}
&\frac{d}{dt}E_2(t)+\mu\int_0^\infty\rho(r^2u_{xt})^2dx+2\mu\int_0^\infty\rho^{-1}\frac{u_t^2}{r^2}dx=-2\frac{Ca}{2}\int_0^\infty(\rho^\gamma)_xruu_tdx\\
&+\frac{Ca}{2}\frac{\gamma(\gamma+1)}{2}\int_0^\infty\rho^{\gamma-4}\rho_t^3dx+4\gamma\frac{Ca}{2}\int_0^\infty\rho^{\gamma-3}\rho_t^2\frac{u}{r}dx+6\gamma\frac{Ca}{2}\int_0^\infty\rho^{\gamma-2}\rho_t\frac{u^2}{r^2}dx\\
&-\mu\int_0^\infty\rho_t(r^2u)_x(r^2u_t)_xdx-\mu\int_0^\infty\rho\left((r^2)_tu\right)_x(r^2u_t)_xdx+\mu\int_0^\infty\left(\rho(r^2u)_x\right)_x(r^2)_tu_tdx\\
&+2\mu(u^2u_t)|_{x=0}-3\gamma_0\left(\frac{Ca}{2}+\frac{2}{We}\right)(R^{-3\gamma_0+1}-1)\frac{dR}{dt}\frac{d^2R}{dt^2}.
\end{aligned}\end{equation}
and
\begin{equation}\begin{aligned}
&\frac{d}{dt}E_3(t)+\frac{Ca}{2}\mu\gamma\int_0^\infty\rho^\gamma r^4(\log\rho)_{xt}^2dx=-2\frac{Ca}{2}\int_0^\infty(\rho^\gamma)_xruu_tdx+\frac{Ca}{2}(\rho^\gamma)_t|_{x=0}R^2\frac{d^2R}{dt^2}\\
&+\frac{Ca}{2}\frac{\gamma(\gamma+1)}{2}\int_0^\infty\rho^{\gamma-4}\rho_t^3dx+4\gamma\frac{Ca}{2}\int_0^\infty\rho^{\gamma-3}\rho_t^2\frac{u}{r}dx+6\gamma\frac{Ca}{2}\int_0^\infty\rho^{\gamma-2}\rho_t\frac{u^2}{r^2}dx\\
&-2\frac{Ca}{2}\mu\int_0^\infty(\rho^\gamma)_xr^3u(\log\rho)_{xt}dx-\frac{Ca}{2}\mu\gamma\int_0^\infty r^4(\rho^\gamma)_x(\log\rho)_t(\log\rho)_{xt}dx.
\end{aligned}\end{equation}
\end{lem}
\begin{proof}
To derive (4.13), we differentiate (1.15) with respect to $t$ and multiply the resulted equation by $u_t$. Using integration by parts, equation (1.14) and (1.18), the term involving pressure is
\begin{equation}\begin{aligned}
\left[(\rho^\gamma)_xr^2\right]_tu_t=&(\rho^\gamma)_x(r^2)_tu_t+(\rho^\gamma)_{xt}r^2u_t\\=&\left[(\rho^\gamma)_tr^2u_t\right]_x-(\rho^\gamma)_t(r^2u_t)_x+2(\rho^\gamma)_xruu_t\\
=&\left[(\rho^\gamma)_tr^2u_t\right]_x+\frac{2\gamma}{(\gamma-1)^2}\partial_t((\rho^\frac{\gamma-1}{2})_t^2)\\
&-\frac{\gamma(\gamma+1)}{2}\rho^{\gamma-4}\rho_t^3-4\gamma\rho^{\gamma-3}\rho_t^2\frac{u}{r}-6\gamma\rho^{\gamma-2}\rho_t\frac{u^2}{r^2}
+2(\rho^\gamma)_xruu_t,
\end{aligned}\end{equation}
and for the term involving viscosity that
$$\begin{aligned}
\left[(\rho(r^2u)_x)_xr^2\right]_tu_t
=&\left[(\rho(r^2u)_x)_tr^2u_t\right]_x-\left[\rho(r^2u)_x\right]_t(r^2u_t)_x
+(\rho(r^2u)_x)_x(r^2)_tu_t\\
=&\left[(\rho r^2u_x)_tr^2u_t+2ru_t^2-2u^2u_t\right]_x-\rho(r^2u_t)_x^2\\
&-\rho_t(r^2u)_x(r^2u_t)_x-\rho((r^2)_tu)_x(r^2u_t)_x+(\rho(r^2u)_x)_x(r^2)_tu_t\\
=&\left[(\rho r^2u_x)_tr^2u_t\right]_x-2(u^2u_t)_x-\rho(r^2u_{xt})^2-2\rho^{-1}\frac{u_t^2}{r^2}\\
&-\rho_t(r^2u)_x(r^2u_t)_x-\rho((r^2)_tu)_x(r^2u_t)_x+(\rho(r^2u)_x)_x(r^2)_tu_t.
\end{aligned}$$
The boundary term, in view of (1.17), is
$$\begin{aligned}
&\left.\left[-\frac{Ca}{2}(\rho^\gamma)_t+\mu(\rho r^2u_x)_t\right]\right|_{x=0}(r^2u_t)|_{x=0}\\
=&-\left[\left(\frac{Ca}{2}+\frac{2}{We}\right)R^{-3\gamma_0}-\frac{2}{We}R^{-1}\right]_tR^2\frac{d^2R}{dt^2}\\
=&\left[\frac{3\gamma_0}{2}\left(\frac{Ca}{2}+\frac{2}{We}\right)-\frac{1}{We}\right]\frac{d}{dt}\left(\frac{dR}{dt}\right)^2+3\gamma_0\left(\frac{Ca}{2}+\frac{2}{We}\right)(R^{-3\gamma_0+1}-1)\frac{dR}{dt}\frac{d^2R}{dt^2}.
\end{aligned}$$
Then by integrating the equation on $(x,t)\in(0,+\infty)\times(0,T)$ and playing the same trick as in Lemma 4.1 (using $\tilde{u}_k$ and (3.13) to apply integration by parts), we arrive at (4.13) .
To derive (4.14), first note that 
$$\left((\rho^\gamma)_xr^2\right)_t(\log\rho)_{xt}r^2=\gamma\rho^\gamma r^4(\log\rho)_{xt}^2+2(\rho^\gamma)_xr^3u(\log\rho)_{xt}+\gamma r^4(\rho^\gamma)_x(\log\rho)_t(\log\rho)_{xt}.$$
Then (4.14) follows from multiplying (2.28) by $(u_t+\mu(\log\rho)_{xt}r^2)$, integrating on $(0,+\infty)$ and  (4.15).
\end{proof}
With the help of Lemma 4.5, one can establish controls for the derivatives of the solution.
\begin{lem}
Suppose $E_0(0)$ and $E_1(0)$ satisfy the same assumption as in Lemma 4.3, then there exists $C>0$ such that\\
(i) $E_2(t)+E_3(t)\leq C(E_2(0)+E_3(0))$ for any $t>0$,\\
(ii) $ \int_0^\infty\int_0^\infty\rho(r^2u_{xt})^2dxdt
+2\int_0^\infty\int_0^\infty\rho^{-1}\frac{u_t^2}{r^2}dxdt\leq C(E_2(0)+E_3(0))$,\\
(iii) $\int_0^\infty\int_0^\infty r^4\rho^\gamma(\log\rho)_{xt}^2dxdt\leq C(E_2(0)+E_3(0))$.
\end{lem}
\begin{proof}
Lemma 4.6 is established by using (4.13)(4.14) and Gronwall's inequality. Therefore, the proof is essentially based on the estimates of the nonlinear terms in (4.13) and (4.14). According to (iii) of Lemma 4.3 and $E_0(0)+E_1(0)\lesssim 1$, $\rho$ is uniformly bounded in the sense that $\rho\leq C$ and $\rho^{-1}\leq C$, which will be used throughout the proof. \\
\textbf{Step 1.} Control of the boundary terms.\\
Similar to Lemma 2.3, we have the $L^\infty$ control for $u_t$ that \begin{equation}
\|u_t^2r\|_{L^\infty}\leq\int_0^\infty\rho(r^2u_{xt})^2dx+2\int_0^\infty\rho^{-1}\frac{u_t^2}{r^2}dx. 
\end{equation}
From 
\begin{equation}
\left|(\rho^\gamma)_t|_{x=0}\right|\leq C\|(\log\rho)_t\|_{L^\infty}\leq C\|\rho^\frac{\gamma}{2} r^2(\log\rho)_{xt}\|_{L^2}\|\rho^\frac{-\gamma}{2}r^{-2}\|_{L^2}\leq C\|\rho^\frac{\gamma}{2} r^2(\log\rho)_{xt}\|_{L^2},
\end{equation}
it follows for $\lambda>0$ to be determined that the boundary term of (4.14) can be bounded by 
\begin{equation}\begin{aligned}
&\left|(\rho^\gamma)_t|_{x=0}R^2\frac{d^2R}{dt^2}\right|\leq C\left|(\rho^\gamma)_t|_{x=0}\right|\|u_tr^\frac{1}{2}\|_{L^\infty}\\
\leq&\lambda\int_0^\infty\rho^\gamma r^4(\log\rho)_{xt}^2dx+C_\lambda\left(\int_0^\infty\rho(r^2u_{xt})^2dx
+2\int_0^\infty\rho^{-1}\frac{u_t^2}{r^2}dx\right).
\end{aligned}\end{equation}
Using $\frac{dR}{dt}=u|_{x=0}$, $r\geq R$, and $C\geq R\geq c$, one has the control for the nonlinear boundary terms in (4.13) that
\begin{equation}\begin{aligned}
\left|(u^2u_t)|_{x=0}\right|\leq\|u_t\|_{L^\infty}u^2|_{x=0}\leq\epsilon\left(\int_0^\infty\rho(r^2u_{xt})^2dx
+2\int_0^\infty\rho^{-1}\frac{u_t^2}{r^2}dx\right)+C_\epsilon \|u^2r\|_{L^\infty}\left(\frac{dR}{dt}\right)^2.
\end{aligned}\end{equation}
Next, using boundary condition (1.16), $\mu\partial_x\left(\rho(r^2u)_x\right)=\frac{u_t}{r^2}+\frac{Ca}{2}(\rho^\gamma)_x$ and $\left|\frac{dR}{dt}\right|=\left|u|_{x=0}\right|\leq\|u\|_{L^\infty}$, we have
\begin{equation}\begin{aligned}
&\left|\left(R^{-3\gamma_0+1}-1\right)\frac{dR}{dt}\frac{d^2R}{dt^2}\right|\\
\leq&\epsilon\|u_t\|_{L^\infty}^2+C_\epsilon\left(\frac{dR}{dt}\right)^2(R-1)^2\\
\leq&\epsilon\|u_t\|_{L^\infty}^2+C_\epsilon\left.\left(\frac{dR}{dt}\right)^2\left(\frac{Ca}{2}(\rho^\gamma-1)-\mu\rho r^2u_x\right)^2\right|_{x=0}\\
\leq&\epsilon\|u_t\|_{L^\infty}^2+C_\epsilon\left(\frac{dR}{dt}\right)^2\left(\|\rho-1\|_{L^\infty}^2+\|\rho(r^2u)_x\|_{L^\infty}^2+\|u^2r\|_{L^\infty}\right)\\
\leq&\epsilon\|u_t\|_{L^\infty}^2+C_\epsilon\left(\frac{dR}{dt}\right)^2\left(\|(\rho^\frac{\gamma}{2})_xr^2\|_{L^2}^2+\|u_t\|_{L^2}^2+\|u^2r\|_{L^\infty}\right)\\
\leq&\epsilon\left(\int_0^\infty\rho(r^2u_{xt})^2dx
+2\int_0^\infty\rho^{-1}\frac{u_t^2}{r^2}dx\right)+C_\epsilon\left(\frac{dR}{dt}\right)^2\left(\|(\rho^\frac{\gamma}{2})_xr^2\|_{L^2}^2+\|u^2r\|_{L^\infty}\right)\\
&+C_\epsilon\|u^2r\|_{L^\infty}\|u_t\|_{L^2}^2.
\end{aligned}\end{equation}\\
\textbf{Step 2.} The rest terms in (4.13).\\
We begin with the estimates of the terms with coefficient $\frac{Ca}{2}$. Using (4.16) and H\"older inequality, one finds
\begin{equation}\begin{aligned}
\left|\int_0^\infty(\rho^\gamma)_xruu_tdx\right|\leq&\epsilon\|u_t^2r\|_{L^\infty}+C_\epsilon\int_0^\infty\rho^{-1}\frac{u^2}{r^2}dx\int_0^\infty(\rho^\gamma)_x^2r^4dx\\
\leq& \epsilon\left(\int_0^\infty\rho(r^2u_{xt})^2dx
+2\int_0^\infty\rho^{-1}\frac{u_t^2}{r^2}dx\right)\\
&+C_\epsilon\int_0^\infty\rho^{-1}\frac{u^2}{r^2}dx\int_0^\infty\left(u_t+\mu (\log\rho)_{xt}r^2\right)^2dx.
\end{aligned}\end{equation}
Use (4.17) and (1.14) to show 
\begin{equation}\begin{aligned}
&\left|\int_0^\infty\rho^{\gamma-4}\rho_t^3dx\right|\\
\leq&\epsilon\|(\log\rho)_t\|_{L^\infty}^2+C_\epsilon\left(\int_0^\infty(\rho^\frac{\gamma-1}{2})_t^2dx\right)^2\\
\leq&\epsilon\int_0^\infty\rho^\gamma r^4(\log\rho)_{xt}^2dx+C_\epsilon\left(\int_0^\infty\rho(r^2u_x)^2dx+
2\int_0^\infty\rho^{-1}\frac{u^2}{r^2}dx\right)\int_0^\infty(\rho^\frac{\gamma-1}{2})_t^2dx,
\end{aligned}\end{equation}
and similarly
\begin{equation}\begin{aligned}
\left|\int_0^\infty\rho^{\gamma-3}\rho_t^2\frac{u}{r}dx\right|\leq\epsilon\int_0^\infty\rho^\gamma r^4(\log\rho)_{xt}^2dx+C_\epsilon\int_0^\infty\rho^{-1}\frac{u^2}{r^2}dx\int_0^\infty(\rho^\frac{\gamma-1}{2})_t^2dx.
\end{aligned}\end{equation}
Using $c\leq R\leq C$ and the equation 
$$2(ru^2)|_{x=0}+\int_0^\infty\rho(r^2u)_x^2dx=\int_0^\infty\rho(r^2u_x)^2dx
+2\int_0^\infty\rho^{-1}\frac{u^2}{r^2}dx,$$ 
we have the inequality
\begin{equation}\begin{aligned}
\int_0^\infty\rho(r^2u_x)^2dx+2\int_0^\infty\rho^{-1}\frac{u^2}{r^2}dx\leq& C\left(\int_0^\infty(\rho^\frac{\gamma-1}{2})_t^2dx+\left(\frac{dR}{dt}\right)^2\right),
\end{aligned}\end{equation}
and thus the last term is controlled by
\begin{equation}\begin{aligned}
&\left|\int_0^\infty\rho^{\gamma-2}\rho_t\frac{u^2}{r^2}dx\right|\\
\leq&\epsilon\int_0^\infty\rho^\gamma r^4(\log\rho)_{xt}^2dx+C_\epsilon\left(\int_0^\infty\rho^{-1}\frac{u^2}{r^2}dx\right)^2\\
\leq&\epsilon\int_0^\infty\rho^\gamma r^4(\log\rho)_{xt}^2dx+C_\epsilon\left(\int_0^\infty\rho^{-1}\frac{u^2}{r^2}dx\right)\left(\int_0^\infty(\rho^\frac{\gamma-1}{2})_t^2dx+\left(\frac{dR}{dt}\right)^2\right).
\end{aligned}\end{equation}
For the terms with coefficient $\mu$, first note that by $\mu\partial_x(\rho(r^2u)_x)=\frac{u_t}{r^2}+\frac{Ca}{2}(\rho^\gamma)_x$, one has
\begin{equation}
\|(\log\rho)_t\|^2_{L^\infty}=\|\rho(r^2u)_x\|^2_{L^\infty}\leq C\left(\|u_t\|_{L^2}^2+\|(\rho^\gamma)_xr^2\|_{L^2}^2\right).
\end{equation}
Then using (4.26), it holds
\begin{equation}\begin{aligned}
&\left|\int_0^\infty\rho_t(r^2u)_x(r^2u_t)_xdx\right|\\
\leq&\epsilon\int_0^\infty\rho(r^2u_t)_x^2dx+C_\epsilon\|\rho(r^2u)_x\|_{L^\infty}^2\int_0^\infty(\rho^\frac{\gamma-1}{2})_t^2dx\\
\leq&\epsilon\left(\int_0^\infty\rho(r^2u_{xt})^2dx
+2\int_0^\infty\rho^{-1}\frac{u_t^2}{r^2}dx\right)+C_\epsilon\left(\|u_t\|_{L^2}^2+\|(\rho^\gamma)_xr^2\|_{L^2}^2\right)\int_0^\infty(\rho^\frac{\gamma-1}{2})_t^2dx\\
\leq&\epsilon\left(\int_0^\infty\rho(r^2u_{xt})^2dx
+2\int_0^\infty\rho^{-1}\frac{u_t^2}{r^2}dx\right)+C_\epsilon\int_0^\infty(\rho^\frac{\gamma}{2})_x^2r^4dx\int_0^\infty(\rho^\frac{\gamma-1}{2})_t^2dx\\
&+C_\epsilon\left(\int_0^\infty\rho(r^2u_x)^2dx
+2\int_0^\infty\rho^{-1}\frac{u^2}{r^2}dx\right)\int_0^\infty u_t^2dx.
\end{aligned}\end{equation}
Using $r_t=u$, (4.24) and $(ru^2)_x=-3\rho^{-1}\frac{u^2}{r^2}+2\frac{u}{r}(r^2u)_x$, it follows
\begin{equation}\begin{aligned}
&\left|\int_0^\infty\rho((r^2)_tu)_x(r^2u_t)_xdx\right|\\
\leq&\epsilon\left(\int_0^\infty\rho(r^2u_{xt})^2dx
+2\int_0^\infty\rho^{-1}\frac{u_t^2}{r^2}dx\right)+C_\epsilon\int_0^\infty\rho(ru^2)_x^2dx\\
\leq&\epsilon\left(\int_0^\infty\rho(r^2u_{xt})^2dx
+2\int_0^\infty\rho^{-1}\frac{u_t^2}{r^2}dx\right)+C_\epsilon
\|u^2r\|_{L^\infty}\left(\int_0^\infty\rho(r^2u)_x^2dx+\int_0^\infty\rho^{-1}\frac{u^2}{r^2}dx\right)\\
\leq&\epsilon\left(\int_0^\infty\rho(r^2u_{xt})^2dx
+2\int_0^\infty\rho^{-1}\frac{u_t^2}{r^2}dx\right)+C_\epsilon
\|u^2r\|_{L^\infty}\left(\int_0^\infty(\rho^\frac{\gamma-1}{2})_t^2dx+\left(\frac{dR}{dt}\right)^2\right).
\end{aligned}\end{equation}
Using equation (1.14), the last term can be controlled easily
\begin{equation}\begin{aligned}
\left|\int_0^\infty\left(\rho(r^2u)_x\right)_x(r^2)_tu_tdx\right|=&2\left|\int_0^\infty(\log\rho)_{xt}r^2\frac{u}{r}u_tdx\right|\\
\leq&\epsilon\int_0^\infty\rho^\gamma r^4(\log\rho)_{xt}^2dx+C_\epsilon\int_0^\infty \frac{u^2}{r^2}u_t^2dx\\
\leq&\epsilon\int_0^\infty\rho^\gamma r^4(\log\rho)_{xt}^2dx+C_\epsilon\|u^2r\|_{L^\infty}\int_0^\infty u_t^2dx.
\end{aligned}\end{equation}
\textbf{Step 3.} The rest terms in (4.14).\\
Since the terms with coefficient $\frac{Ca}{2}$ are the same as in (4.13), it suffices to estimate the rest two terms.
Noting that $\frac{Ca}{2}(\rho^\gamma)_xr^2=u_t+\mu r^2(\log\rho)_{xt}$, it holds by (4.26) that
\begin{equation}\begin{aligned}
&\left|\int_0^\infty r^4(\rho^\gamma)_x(\log\rho)_t(\log\rho)_{xt}dx\right|\\
\leq&\epsilon\int_0^\infty\rho^\gamma r^4(\log\rho)_{xt}^2dx+C_\epsilon\|(\log\rho)_t\|_{L^\infty}^2\int_0^\infty(\rho^\gamma)_x^2r^4dx\\
\leq&\epsilon\int_0^\infty\rho^\gamma r^4(\log\rho)_{xt}^2dx+C_\epsilon\int_0^\infty(\rho^\frac{\gamma}{2})_x^2 r^4dx\left(\int_0^\infty u_t^2dx+\int_0^\infty\left(u_t+\mu r^2(\log\rho)_{xt}\right)^2dx\right).
\end{aligned}\end{equation}
The other term can be bounded easily
\begin{equation}\begin{aligned}
&\left|\int_0^\infty(\rho^\gamma)_xr^3u(\log\rho)_{xt}dx\right|\leq\epsilon\int_0^\infty\rho^\gamma r^4(\log\rho)_{xt}^2dx+C_\epsilon\|u^2r\|_{L^\infty}\int_0^\infty\left(u_t+\mu r^2(\log\rho)_{xt}\right)^2dx.
\end{aligned}\end{equation}
Now, choose positive $\lambda>0$, $A>0$ such that $\lambda\leq\frac{\mu\gamma}{3}$ and $A\geq 3\mu^{-1}\frac{Ca}{2}C_\lambda$. Then by collecting (4.18)-(4.31), we find that
\begin{equation}\begin{aligned}
&\frac{d}{dt}(AE_2+E_3)+\frac{2A\mu}{3}\left(\int_0^\infty\rho(r^2u_{xt})^2dx+2\int_0^\infty\rho^{-1}\frac{u_t^2}{r^2}dx\right)+\frac{Ca}{3}\mu\gamma\int_0^\infty\rho^\gamma r^4(\log\rho)_{xt}^2dx\\
\leq&\epsilon(A+1)\left(\int_0^\infty\rho(r^2u_{xt})^2dx+2\int_0^\infty\rho^{-1}\frac{u_t^2}{r^2}dx+\int_0^\infty\rho^\gamma r^4(\log\rho)_{xt}^2dx\right)\\
&+C_\epsilon(A+1)\left(\int_0^\infty\rho(r^2u)_x^2dx+\int_0^\infty\rho^{-1}\frac{u^2}{r^2}dx+\int_0^\infty r^4(\rho^\frac{\gamma}{2})_x^2dx\right)(AE_2+E_3).
\end{aligned}\end{equation}
Choose $\epsilon>0$ small such that $\epsilon(A+1)\leq\min\left\{\frac{A\mu}{3},\frac{Ca}{6}\mu\gamma\right\}$. Then (4.32) becomes 
\begin{equation}\begin{aligned}
&\frac{d}{dt}(AE_2+E_3)+\frac{A\mu}{3}\left(\int_0^\infty\rho(r^2u_{xt})^2dx+2\int_0^\infty\rho^{-1}\frac{u_t^2}{r^2}dx\right)+\frac{Ca}{6}\mu\gamma\int_0^\infty\rho^\gamma r^4(\log\rho)_{xt}^2dx\\
\leq& C\left(\int_0^\infty\rho(r^2u)_x^2dx+\int_0^\infty\rho^{-1}\frac{u^2}{r^2}dx+\int_0^\infty r^4(\rho^\frac{\gamma}{2})_x^2dx\right)(AE_2+E_3).
\end{aligned}\end{equation}
 Hence (i)(ii)(iii) follow from applying Gronwall's inequality to (4.33) in view of Lemma 4.1 and Lemma 4.3.
\end{proof}
\begin{cor} Suppose $E_0(0)$ and $E_1(0)$ satisfy the same assumptions as in Lemma 4.3. Then there exists $C>0$ such that\\
(i) $\int_0^\infty E_3(t)dt\leq C(E_0(0)+E_1(0))$,\\
(ii) $\int_0^\infty E_2(t)dt\leq C(E_0(0)+E_1(0)+E_2(0)+E_3(0))$.
\end{cor}
\begin{proof}
The equation (1.15) yields that
$$
\int_0^\infty\int_0^\infty(u_t+\mu r^2(\log\rho)_{xt})^2dxdt=\left(\frac{Ca}{2}\right)^2\int_0^\infty\int_0^\infty(\rho^\gamma)_x^2r^4dxdt.
$$
Hence by (ii)(iii) of Lemma 4.3,
\begin{equation}
\int_0^\infty\int_0^\infty(u_t+\mu r^2(\log\rho)_{xt})^2dxdt\leq C\int_0^\infty\int_0^\infty(\rho^\frac{\gamma}{2})_x^2 r^4dxdt
\leq C(E_0(0)+E_1(0)).
\end{equation}
Meanwhile, using Lemma 4.1 and $\rho\approx 1$ again, it follows from 
$$\int_0^\infty\int_0^\infty\rho(r^2u)_x^2dxdt
\leq C\left(\int_0^\infty\int_0^\infty\rho(r^2u_x)^2dxdt+
2\int_0^\infty\int_0^\infty\rho^{-1}\frac{u^2}{r^2}dxdt\right)$$
that
\begin{equation}\begin{aligned}
\int_0^\infty\int_0^\infty(\rho^\frac{\gamma-1}{2})_t^2dxdt\leq C\int_0^\infty\int_0^\infty\rho(r^2u)_x^2dxdt
\leq CE_0(0).
\end{aligned}\end{equation}
(i) is a direct consequence of (4.34) and (4.35). To show (ii), write $u_t=(u_t+\mu r^2(\log\rho)_{xt})-\mu r^2(\log\rho)_{xt}$. Then (ii) follows from (4.34)(4.35), (iii) of Lemma 4.6 and $\rho\approx 1$.
\end{proof}
To finish the proof of Theorem 1.3, the following lemma is necessary.
\begin{lem}
Suppose that $E(t)>0$ and $\alpha(t)>0$ such that $\int_0^\infty E(t)dt\leq+\infty$, $\int_0^\infty\alpha dt\leq+\infty$, $E(0)<+\infty$ and $\frac{d}{dt}E\leq \alpha E$. Then there exists constant $C>0$ such that $E(t)\leq C(1+t)^{-1}$.
\end{lem}
\begin{proof}
A direct computation gives $\frac{d}{dt}\left((1+t)E\right)\leq\alpha (1+t)E+E$. Then Gronwall's inequality yields that $(1+t)E(t)\leq \exp(\int_0^\infty\alpha dt)\left(E(0)+\int_0^\infty Edt\right)$.
\end{proof}
\noindent \textbf{Proof of Theorem 1.3.}
With the help of Lemma 4.1, Lemma 4.3 and corollary 4.7, applying Lemma 4.8 to (4.33) yields that
$$
\|u_t\|_{L^2}^2+\|r^2\rho_x\|_{L^2}^2+\|\rho_t\|_{L^2}^2+\left(\frac{dR}{dt}\right)^2\leq C(1+t)^{-1}.
$$
Hence the decay of $\|r^2\rho_x\|_{L^2}$, $\|r^2u_x\|_{L^2}$ and $\left\|\frac{u}{r}\right\|_{L^2}$ follow from (4.24). To prove (1.22), it remains to show $(R-1)^2\lesssim (1+t)^{-1}$. In view of $R\approx 1$ and $\rho\approx 1$, (1.17) gives that
$$(R-1)^2\leq C\left(\|\rho-1\|_{L^\infty}^2+\|(\log\rho)_t\|_{L^\infty}^2+\left(\frac{dR}{dt}\right)^2\right).$$
The first two terms on the right-hand side are controlled by using Sobolev embedding $\|\rho-1\|_{L^\infty}^2\lesssim\|r^2\rho_x\|_{L^2}^2$ and $\|(\log\rho)_t\|_{L^\infty}^2\lesssim\|r^2(\log\rho)_{xt}\|_{L^2}^2\lesssim\|r^2\rho_x\|_{L^2}^2+\|u_t\|_{L^2}^2$. Since $\|r^2\rho_x\|_{L^2}^2$, $\|u_t\|_{L^2}^2$ and $\left(\frac{dR}{dt}\right)^2$ all decay with speed $(1+t)^{-1}$, it follows $(R-1)^2\lesssim (1+t)^{-1}$.
\bibliography{Bubble}
\end{document}